\documentclass[a4paper,12pt]{scrartcl}

\usepackage[english]{babel}
\usepackage[utf8]{inputenc}
\usepackage{amsmath, amsthm, amssymb}
\usepackage[pdftex]{graphicx,color}
\usepackage{float}
\usepackage{caption}
\usepackage{subcaption}
\usepackage[dvipsnames]{xcolor} 
\usepackage{tikz,pgf}
\usetikzlibrary{calc}
\usepackage[colorlinks=true,linkcolor=RoyalBlue,citecolor=PineGreen]{hyperref}

\newcommand{\blue}{\text{blue}}
\newcommand{\red}{\text{red}}

\newcommand{\norm}[1]{\left\lVert#1\right\rVert}

\newcommand{\trcon}{$\triangle$-connected}
\newcommand{\trequiv}{\sim_{\!\!\bigtriangleup}}
\newcommand{\nottrequiv}{\nsim_{\!\!\bigtriangleup}}

\newcommand{\NN}{\mathbb{N}}
\newcommand{\RR}{\mathbb{R}}
\newcommand{\CC}{\mathbb{C}}
\newcommand{\ci}{\mathrm{i}}

\newtheorem{thm}{Theorem}[section]
\newtheorem{lem}[thm]{Lemma}
\newtheorem{cor}[thm]{Corollary}
\newtheorem{prop}[thm]{Proposition}

\theoremstyle{definition}
\newtheorem{defn}[thm]{Definition}
\newtheorem{exmp}[thm]{Example}
\newtheorem{conj}[thm]{Conjecture}

\tikzstyle{vertex}=[circle, draw, fill=black, inner sep=0pt, minimum size=4pt]
\tikzstyle{edge}=[line width=1.5pt,black!50!white]
\tikzstyle{gridp}=[inner sep=1pt,circle,fill=black!70!white]
\tikzstyle{gridl}=[black!50!white]

\tikzstyle{lnode}=[circle,white,draw=black!60!white,fill=black!60!white,inner sep=1pt]
\tikzstyle{cnode}=[circle,draw=black!60!white,fill=black!60!white,inner sep=1.5pt]
\tikzstyle{redge}=[edge,Red]
\tikzstyle{bedge}=[edge,NavyBlue]

\colorlet{ncol}{Green!60!black}
\tikzstyle{nvertex}=[vertex, draw=ncol, fill=ncol]
\tikzstyle{edgeq}=[edge,gray!60,densely dashed]
\tikzstyle{nedge}=[edge,ncol]
\tikzstyle{oedge}=[edge,Red!60!black]

\begin{document}

\title{Graphs with Flexible Labelings}

\author{%
Georg Grasegger\thanks{Johann Radon Institute for Computational and Applied Mathematics (RICAM), Austrian Academy of Sciences}
\and 
Jan Legersk\'y\thanks{Research Institute for Symbolic Computation (RISC), Johannes Kepler University Linz}
\and 
Josef Schicho\footnotemark[2]
}

\date{}

\maketitle

\footnotetext{
	\mbox{}\\[0.5ex]
	\begin{minipage}[t]{0.8\textwidth}
		This project has received funding from the European Union's Horizon~2020 research and innovation programme under the Marie Sk\l{}odowska-Curie grant agreement No~675789.
	\end{minipage}\hfill 
	\begin{minipage}[t]{0.12\textwidth}
		\vspace{-1em}
		\includegraphics[width=0.95\textwidth]{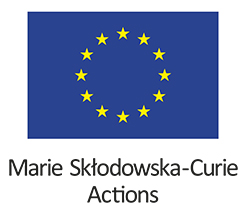}
	\end{minipage}
	Partially supported by the Austrian Science Fund (FWF): P26607, W1214-N15 (project DK9); and by the Upper Austrian Government
}

\begin{abstract}
For a flexible labeling of a graph, it is possible to construct infinitely many non-equivalent
realizations keeping the distances of connected points constant. We give a combinatorial characterization
of graphs that have flexible labelings, possibly non-generic. The characterization is based on colorings of the edges
with restrictions on the cycles. Furthermore, we give necessary criteria and sufficient ones for the existence of such colorings.
\end{abstract}

\section{Introduction}
Given a graph together with a labeling of its edges by positive real numbers, we are interested in the
set of all functions from the set of vertices to the real plane such that the distance between any
two connected vertices is equal to the label of the edge connecting them. Apparently, any such
``realization'' gives rise to infinitely many equivalent ones, by the action of the group of 
Euclidean congruence transformations. If the set of equivalence classes is finite and non-empty, then we say
that the labeled graph is rigid; if this set is infinite, then we say that the labeling is
flexible.

The main result in this paper is a combinatorial characterization of graphs that have a flexible
labeling. The characterization of graphs such that a generically chosen assignment of the vertices 
to the real plane gives a flexible labeling is classical: by a theorem of Geiringer-Pollaczek~\cite{Geiringer1927} which was rediscovered by Laman~\cite{Laman1970},
this is true if the graph contains no Laman subgraph with the same set of vertices. Here, a graph
$G=(V_G,E_G)$ is called a Laman graph if and only if $|E_G|=2|V_G|-3$ and $|E_H|\le 2|V_H|-3$ for any subgraph 
$H=(V_H,E_H)$ of~$G$. So, if either $|E_G|<2|V_G|-3$, or if $|E_G|=2|V_G|-3$ and there is a subgraph~$H$ 
such that $|E_H|>2|V_H|-3$, then there exists a flexible labeling. 

However, there exist also many Laman graphs that have flexible labelings.
These labelings are necessarily non-generic.
An example is the 3-prism with all lengths equal (see Figure~\ref{fig:3prism}).
\begin{figure}
  \begin{center}
    \pgfdeclarelayer{background}
		\pgfsetlayers{background,main}
    \begin{tikzpicture}[scale=2]
      \node[vertex] (a) at (0,0) {};
      \node[vertex] (b) at (1,0) {};
      \node[vertex] (c) at (0.5,0.5) {};
      \node[vertex] (d) at (0,1) {};
      \node[vertex] (e) at (1,1) {};
      \node[vertex] (f) at (0.5,1.5) {};
      
      \draw[edge] (a)edge(b) (b)edge(c) (c)edge(a) (a)edge(d) (d)edge(e) (e)edge(f) (f)edge(d) (b)edge(e) (c)edge(f);
      \begin{scope}[]
				\node[vertex] (d2) at (0.6,0.8) {};
				\node[vertex] (e2) at (1.6,0.8) {};
				\node[vertex] (f2) at (1.1,1.3) {};
				
				\draw[edge,black!20!white] (a)edge(b) (b)edge(c) (c)edge(a) (a)edge(d2) (d2)edge(e2) (e2)edge(f2) (f2)edge(d2) (b)edge(e2) (c)edge(f2);
      \end{scope}
      \begin{pgfonlayer}{background}
        \draw[dotted,black!40!white,line width=2pt] (d) arc[start angle=90,end angle=53.1301,radius=1cm] (d2);
        \draw[dotted,black!40!white,line width=2pt] (e) arc[start angle=90,end angle=53.1301,radius=1cm] (e2);
        \draw[dotted,black!40!white,line width=2pt] (f) arc[start angle=90,end angle=53.1301,radius=1cm] (f2);
      \end{pgfonlayer}
    \end{tikzpicture}
  \end{center}
  \caption{A flexible labeling of the 3-prism graph: if the edges are labeled by the edge lengths, then there are infinitely
  many other realizations where the distances are equal to these labels.}
  \label{fig:3prism}
\end{figure}
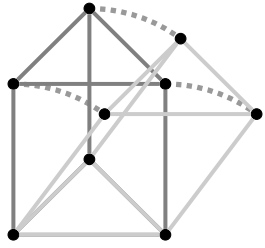
Another example are the bipartite graphs: Dixon~\cite{Dixon} has shown that there are two types of flexible realizations
(see also~\cite{Maehara2001,Stachel,WalterHusty}). 
Flexible instances of graphs in different context are as well considered in~\cite{Fekete2015,Jackson2015,Stachel}.

There are also graphs with more than $2|V_G|-3$ edges allowing flexible labelings.
The maximal number of edges 
for a graph with~$n$ vertices that has a flexible labeling is $\frac{n^2-3n+4}{2}$
(see Section~\ref{sec:examples} and Figure~\ref{fig:manyedges} for examples). 
\begin{figure} 
  \begin{center}
    \begin{tikzpicture}[scale=1.7]
      \node[vertex] (a) at (0.587785, -0.809017) {};
      \node[vertex] (b) at (0.951057, 0.309017) {};
      \node[vertex] (c) at (0., 1.) {};
      \node[vertex] (d) at (-0.951057, 0.309017) {};
      \node[vertex] (e) at (-0.587785, -0.809017) {};
      \node[vertex] (f) at (0,2) {};
      
      \draw[edge] (a)edge(b) (a)edge(c) (a)edge(d) (a)edge(e) (b)edge(c) (b)edge(d) (b)edge(e) (c)edge(d) (c)edge(e) (d)edge(e) (c)edge(f);
      
      \begin{scope}[xshift=3cm]
        \node[vertex] (a) at (0., -1.) {};
				\node[vertex] (b) at (0.951057, -0.309017) {};
				\node[vertex] (c) at (0.587785, 0.809017) {};
				\node[vertex] (d) at (-0.587785, 0.809017) {};
				\node[vertex] (e) at (-0.951057, -0.309017) {};
				\node[vertex] (f) at (0,1.84582) {};
				
				\draw[edge] (a)edge(b) (a)edge(c) (a)edge(d) (a)edge(e) (b)edge(c) (b)edge(d) (b)edge(e) (c)edge(e) (d)edge(e) (d)edge(f) (c)edge(f);
      \end{scope}
    \end{tikzpicture}
  \end{center}
  \caption{Two examples of graphs with flexible labeling on 6 vertices and 11 edges (which is the maximal number). To see that the right graph
    has a flexible labeling, one has to consider realizations for which the two vertices connected to the top vertex are mapped to the
    the same point in the plane.}
  \label{fig:manyedges}
\end{figure}
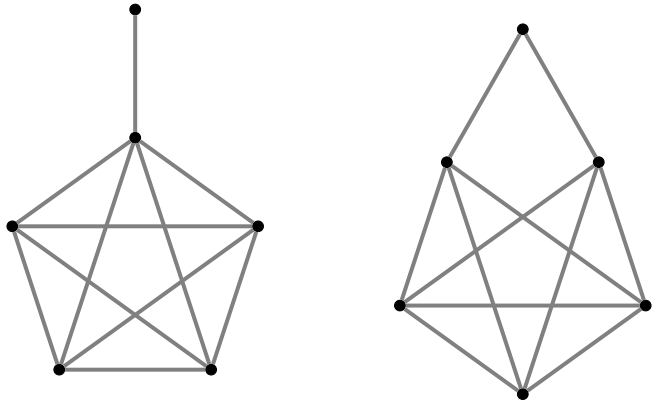

The main part of our paper deals with general graphs,
but for some sections we specifically consider Laman graphs (minimally generically rigid).
We refer to~\cite{Graver1993} for comparing different notions of rigidity.
There are also various notions of genericity (see~\cite{Streinu2008} for a summary),
whereas flexible labelings within this paper are not assumed to be generic.

In order to characterize graphs with a flexible labeling, we introduce the concept of a NAC-coloring (NAC is short for ``no almost cycle''). 
This is a coloring of the edges by two colors, \red{} and \blue{}, such that there is no cycle with all but one edges 
in the same color. We show that any graph with a flexible labeling does have a NAC-coloring.
Conversely, for any graph and any NAC-coloring, we construct a flexible labeling. In Section~\ref{sec:examples}, 
we give various necessary and sufficient conditions on a graph for the existence of a flexible labeling. Section~\ref{sec:conjecture}
gives a conjecture on the existence of flexible labelings in the case of a Laman graph, with some partial answers supporting
the conjecture.
In contrast to many difficult open problems about three-dimensional rigidity,
the question of existence of a flexible labeling turns out to be rather easy, see Section~\ref{sec:threeDim}.

\section{Labelings of Graphs and their Realizations}

In this section, we define labelings of graphs, realizations of labeled graphs, and three properties of graphs:
realizable, rigid, and flexible. We also define NAC-colorings of graphs and give a lemma providing a convenient way 
to decide whether a given coloring is NAC.

In the paper, we consider only undirected graphs without multiple edges or self-loops.
The set of vertices of a graph~$G$ is denoted by~$V_G$ and the set of edges by~$E_G$.
We write~$e=uv$ for an edge~$e\in E_G$ connecting vertices~$u$ and~$v$ in~$V_G$ for shorter notation,
but the fact that the edge can be viewed as the set~$\{u,v\}$ is also used.

An edge labeling by positive real numbers is used to prescribe distances between vertices connected by an edge.
A \emph{realization} is a drawing of~$G$ into the plane satisfying given constraints.
Obviously, we can obtain infinitely many realizations from a given one by rotations and translations.
Let us formalize these concepts.
\begin{defn}
	Let~$G$ be a graph such that $|E_G|\geq 1$ and let $\lambda\colon E_G\rightarrow \RR_+$ be an edge labeling of~$G$.
	A map $\rho=(\rho_x,\rho_y)\colon V_G\rightarrow \RR^2$ is a \emph{realization} of~$G$ compatible
	with~$\lambda$ iff $\norm{\rho(u)-\rho(v)}=\lambda(uv)$ for all edges~$uv\in E_G$.
	We say that two realizations~$\rho_1$ and~$\rho_2$ are equivalent
	iff there exists a direct Euclidean isometry~$\sigma$ of~$\RR^2$ such that $\rho_1=\sigma \circ\rho_2$.
\end{defn}
We stress that a realization is not required to be injective -- nonadjacent vertices can overlap.
Nevertheless, adjacent vertices need to be mapped to different points, since lengths of edges are positive.

Definition of rigid and flexible labelings are well established though they exist in various versions.
Here we present the one used within this paper.
\begin{defn}
	A labeling $\lambda\colon E_G\rightarrow \RR_+$ of a graph~$G$ is called
	\begin{enumerate}
		\item \emph{realizable} if there is a realization of~$G$ compatible with~$\lambda$,
		\item \emph{rigid} if it is realizable and the number of realizations of~$G$ compatible with~$\lambda$ up to equivalence is finite,
		\item \emph{flexible} if the number of realizations of~$G$ compatible with~$\lambda$ up to equivalence is infinite.
	\end{enumerate}
\end{defn}

The constraints on distances between vertices given by~$\lambda$ can also be expressed
in terms of algebraic equations in variables~$x_u, y_u$ for~$u\in V_G$, which represent the coordinates of~$u$.
In order to choose only one representative of each equivalence class, we choose one edge~$\bar{u}\bar{v}\in E_G$ and we fix its position.
The system of equations is
\begin{align} \label{eq:mainSystemOfEquations}
	x_{\bar{u}}=0\,, \quad	y_{\bar{u}}&=0\,, \nonumber \\
	x_{\bar{v}}=\lambda(\bar{u}\bar{v})\,, \quad y_{\bar{v}}&=0\,, \\
	(x_u-x_v)^2+(y_u-y_v)^2&= \lambda(uv)^2 \quad \text{ for all } uv \in E_G.	\nonumber
\end{align}		
As in the previous definition, the labeling~$\lambda$ is realizable, resp.\ rigid, resp.\ flexible,
if the number of solutions $(x_u,y_u)_{u\in V_G} \in (\RR^2)^{V_G}$ of the system~\eqref{eq:mainSystemOfEquations} is positive,
resp.\ finite and positive, resp.\ infinite. The realization~$\rho$ corresponding to a solution $(x_u,y_u)_{u\in V_G}$ is given by~$\rho(u)=(x_u,y_u)$.

The main result of this paper is that existence of a~$\lambda$ for which \eqref{eq:mainSystemOfEquations}
has infinitely many solutions can be characterized by a coloring of~$G$ with a certain property that is introduced now.
We remark that there is no requirement on colors of adjacent edges in contrary to edge colorings often used in graph theory.
\begin{defn}
	Let~$G$ be a graph and $\delta\colon  E_G\rightarrow \{\text{\blue{}, \red{}}\}$ be a coloring of edges.
	\begin{enumerate}
		\item A path, resp.\ cycle, in~$G$ is a \emph{\red{} path}, resp.\ \emph{\red{} cycle}, if all its edges are \red.
		\item A cycle in~$G$ is an \emph{almost \red{} cycle}, if exactly one of its edges is \blue{}.
		\item The notions of \emph{\blue{} path, \blue{} cycle} and \emph{almost \blue{} cycle} are defined analogously.
	\end{enumerate}
	A coloring~$\delta$ is called a \emph{NAC-coloring}, if it is surjective and there are no almost \blue{} cycles or almost \red{} cycles in~$G$.
\end{defn}
The shortcut NAC stands for "No Almost Cycles". We will see later how this property naturally arises for graphs with a flexible labeling.
Lemma~\ref{lem:equivCycleCond} below provides a characterization that can be verified easily.
Two subgraphs of $G$ induced by a coloring~$\delta$ are used, namely, $G^\delta_\red{}=(V_G,E^\delta_\red{})$ and $G^\delta_\blue{}=(V_G,E^\delta_\blue{})$,
where $E^\delta_\red{}=\{e\in E_G \colon \delta(e)=\red{}\}$ and
$E^\delta_\blue{}=\{e\in E_G \colon \delta(e)=\blue{}\}$.

\begin{lem}\label{lem:equivCycleCond}
	Let~$G$ be a graph. If $\delta\colon E_G\rightarrow \{\blue{}, \red{}\}$ is a coloring of edges,
	then there are no almost \blue{} cycles or almost \red{} cycles in~$G$ if and only if
	the connected components of $G^\delta_\red{}$ and $G^\delta_{\blue{}}$ are induced subgraphs of~$G$.
\end{lem}
\begin{proof}
	$\implies$:
		Let $M$ be a connected component of $G^\delta_\red{}$.
		Assume that $M$ is not an induced subgraph of~$G$, i.e.,
		there exists an edge~$uv \in E_G$ such that~$u,v\in V_M$ but~$uv \notin E_M$.
		Hence, $uv$ is \blue{} and there exists an almost \red{} cycle
		which is a contradiction.
		This works analogously for  connected components of $G^\delta_\blue{}$.

	$\impliedby$:
		Assume that there exists an almost \red{} cycle with a \blue{} edge~$uv$.
		The vertices~$u$ and~$v$ belong to the same connected component of $G^\delta_\red{}$ but~$uv$ is \blue{},
		i.e., $uv\notin E^\delta_\red{}$ which contradicts that all connected components
		of $G^\delta_\red{}$ are induced subgraphs of~$G$.
		Similarly for an almost \blue{} cycle.
\qed\end{proof}

Using Lemma~\ref{lem:equivCycleCond},
there is a straightforward algorithm checking whether a given coloring $\delta$ is a NAC-coloring:
we construct $G^\delta_\red{}$ and $G^\delta_\blue{}$.
If $G^\delta_\red{}$ or $G^\delta_\blue{}$ has no edges,
then $\delta$ is not surjective, thus, it is not a NAC-coloring.
Otherwise, we decompose $G^\delta_\red{}$ and $G^\delta_\blue{}$ into connected components
and check whether these components are induced subgraphs of~$G$.

\section{Flexibility and NAC-colorings}

In this section, we state and prove our main result: a connected graph has a flexible labeling if and only if it has a NAC-coloring. Both
directions of the proof contain a construction: for a given curve of realizations of a graph with a fixed flexible labeling, we construct
a NAC-coloring of the graph. In the other direction, for a given NAC-coloring of a graph, we construct a flexible
labeling and a curve of realizations.

\begin{thm}\label{thm:nacflexible}
	A connected graph~$G$ with at least one edge has a flexible labeling iff it has a NAC-coloring.
\end{thm}
\begin{proof}
	$\implies$:
		We recall that if $W$ is an additive ordered Abelian group, then a \emph{valuation} of a function field $F$ 
		(a finitely generated extension over $\CC$ of transcendence degree at least one) 
		is a nontrivial surjective mapping $\nu\colon F^* \rightarrow W$ such that for all $a,b\in F^*$:
		\begin{enumerate}
			\item $\nu(ab)=\nu(a)+\nu(b)$, and
			\item if $a+b\in F^*$, then $\nu(a+b)\geq \min(\nu(a),\nu(b))$.
		\end{enumerate}
		A direct consequence of Chevalley's Extension Theorem is that if $z\in F$ is transcendental over $\CC$, 
		then there exists at least one valuation~$\nu$ such that~$\nu(z)>0$ and~$\nu(\CC)=\{0\}$. See for instance \cite{Deuring}.
		
		Let $\lambda\colon E_G\rightarrow \RR_+$ be a flexible labeling of~$G$. 
                Let $X\subset (\RR^2)^{V_G}$ be the set of solutions of~\eqref{eq:mainSystemOfEquations}, i.e.,
                the set of all realizations compatible with~$\lambda$ with one  edge fixed.
                Let $C\subseteq X$ be an irreducible algebraic curve contained in~$X$.
                Let $F$ be the complex function field of $C$, which is defined as the set of all partially defined functions from $C$ to $\CC$ that
                can be locally defined as a quotient of polynomials.
		For every $e=uv\in E_G$, we define $W_{e}, Z_{e} \in F$ by
		\begin{align*}
			W_{e}&=(x_u-x_v) + \ci (y_u-y_v)\,, \\
			Z_{e}&=(x_u-x_v) - \ci (y_u-y_v)\,.
		\end{align*}
		Hence
		\[
		W_e Z_e=(x_u-x_v)^2+(y_u-y_v)^2 = \lambda(uv)^2 \in \RR_+\,.
		\]
		There exists an edge $e' \in E_G$ such that $W_{e'}$ is transcendental over $\CC$. 
		Otherwise, there are only finitely many complex realizations of~$G$ and hence also only finitely many real ones, 
		which would contradict that~$\lambda$ is flexible.
		Let~$\nu$ be a valuation of $F$ such that~$\nu(W_{e'})>0$. For all edges $e\in E_G$, we have $\nu(W_e)+\nu(Z_e)=0$.
		
		We define a NAC-coloring:
		\[
			\delta(e)=\begin{cases}
				\red{} &\text{if } \nu(W_e)>0\,, \\
				\blue{} &\text{if } \nu(W_e)\leq 0\,.
			\end{cases}
		\]
		It is surjective as  $e'$ is \red{} and $\bar{u}\bar{v}$ is \blue{}, since $W_{\bar{u}\bar{v}}\in \CC$. We consider 
		a cycle $(u_1, u_2, \dots, u_n)$. 
		If we assume that all edges $u_1u_2, \dots, u_{n-1}u_n$ are \red{}, then
		\[
			\nu(W_{u_1u_n})=\nu(W_{u_1u_2} +W_{u_2u_3}+ \ldots +W_{u_{n-1}u_n})\geq \min \nu(W_{u_i u_{i+1}})> 0\,.
		\]
		Hence, there is no almost \red{} cycle. On the other hand, if $u_1u_2, \dots, u_{n-1}u_n$ are all \blue{}, then 
		\begin{align*}
			\nu(W_{u_1u_n})&=-\nu(Z_{u_1u_n})=-\nu(Z_{u_1u_2} + \ldots +Z_{u_{n-1}u_n})\\
			&\leq -\min \nu(Z_{u_i u_{i+1}})=\max \nu(W_{u_i u_{i+1}})\leq 0\,,
		\end{align*}
		which implies that~$u_1u_n$ is also blue.
		
	$\impliedby$:
		Let~$\delta$ be a NAC-coloring.
		Let $R_1, \dots, R_m$ be the sets of vertices of connected components of the graph $G^\delta_\red{}$ and
		$B_1, \dots, B_n$ be the sets of vertices of connected components of the graph $G^\delta_\blue{}$.

		For $\alpha\in [0,2\pi)$, we define a map $\rho_\alpha\colon V_G\rightarrow \RR^2$ by 
		\begin{equation*}
		  \rho_\alpha(v)=i\cdot (1,0) + j\cdot (\cos\alpha, \sin\alpha)\,,
		\end{equation*}
		where $i$ and $j$ are such that~$v\in R_i \cap B_j$.
		This means we consider a (slanted) grid of coordinates for the vertices of the graph.
		\begin{center}
			\begin{tikzpicture}[]
				\begin{scope}[xshift=0cm]
					\foreach \j in {0,1,2,3}
					{
						\draw[gridl] ($\j*(0,1)$)
						\foreach \i in {0,1,2,3,4}
						{
							-- +($\i*(1,0)$)
						};
					}
					\foreach \i in {0,1,2,3,4}
					{
						\draw[gridl] ($\i*(1,0)$) node[gridp] {}
						\foreach \j in {0,1,2,3}
						{
							-- +($\j*(0,1)$) node[gridp] {}
						};
					}		  
				\end{scope}
				\begin{scope}[xshift=6cm]
					\foreach \j in {0,1,2,3}
					{
						\draw[gridl] ($\j*(0.5, 0.866025)$)
						\foreach \i in {0,1,2,3,4}
						{
							-- +($\i*(1,0)$)
						};
					}
					\foreach \i in {0,1,2,3,4}
					{
						\draw[gridl] ($\i*(1,0)$) node[gridp] {}
						\foreach \j in {0,1,2,3}
						{
							-- +($\j*(0.5, 0.866025)$) node[gridp] {}
						};
					}
				\end{scope}
			\end{tikzpicture}
		\end{center}

		The map $\rho_\alpha$ gives a labeling $\lambda_\alpha\colon E_G\rightarrow \RR_+$, where
		\begin{equation*}
		  \lambda_\alpha(uv)=\norm{\rho_\alpha(u)-\rho_\alpha(v)}
		\end{equation*}
		for all edges~$uv\in E_G$. We show that~$\lambda_\alpha$ is actually independent on the choice of $\alpha$.
		Let~$uv$ be an edge in~$E_G$ and $i,j,k,l$ be such that~$u\in R_i\cap B_j$ and~$v\in R_k\cap B_l$.
		If~$uv$ is \red{}, then $i=k$ and we have
		\begin{align*}
			\lambda_\alpha(uv)&=\norm{\rho_\alpha(u)-\rho_\alpha(v)}=\norm{(i-k)\cdot (1,0) + (j-l)\cdot (\cos\alpha, \sin\alpha)} \\
											  &=|j-l|\cdot \norm{(\cos\alpha, \sin\alpha)}=|j-l|\,.
		\end{align*}
		If~$uv$ is \blue{}, then $j=l$ and hence $\lambda_\alpha(uv)=|k-i|$. 

		Therefore, $\rho_\alpha$ is a realization of~$G$ compatible with $\lambda=\lambda_\frac{\pi}{2}$ for all~$\alpha\in [0,2\pi)$.
		
		We also need to show that the codomain of~$\lambda$ is $\RR_+$.
		In contrary, assume that $\lambda(uv)= 0$ for some~$uv\in E_G$.
		This implies that both~$u$ and~$v$ belong to the same $R_i\cap B_j$.
		This contradicts that~$\delta$ is a NAC-coloring since if~$uv$ is \red{},
		then there is an almost \blue{} cycle and similarly if~$uv$ is \blue{},
		then there is an almost \red{} cycle.
		
		The labeling~$\lambda$ is flexible since there are infinitely many realizations $\rho_\alpha$ which are non equivalent as~$\delta$ is surjective,
		i.e., there are edges in direction $(0,1)$ and also $(\cos\alpha, \sin\alpha)$.
\qed\end{proof}
We illustrate the proof by examples for both implications.
Example~\ref{ex:quadrilateral} shows how a NAC-coloring is obtained from a flexible labeling of $C_4$.
Notice that a generic quadrilateral is flexible, but there is also a rigid case with non-generic labeling 
-- the lengths of three edges sum up to the length of the fourth one.

\begin{exmp}\label{ex:quadrilateral}
	Let~$G$ be a cycle graph with $V_G=\{v_1,v_2,v_3,v_4\}$ and $E_G=\{v_1v_2,v_2v_3,v_3v_4,v_4v_1\}$, see Figure~\ref{fig:quadrilateral}.
	\begin{figure}[ht]
	  \begin{center}
	    \begin{tikzpicture}[scale=2]
	      \node[vertex] (a) at (0,0) {};
				\node[vertex] (b) at (2,0) {};
				\node[vertex] (c) at (1.8,0.8) {};
				\node[vertex] (d) at (0.2,1) {};
				
				\draw[edge] (a) to node[below] {1} (b);
				\draw[edge] (b) to node[right] {$\lambda_{23}$} (c);
				\draw[edge] (c) to node[above] {$\lambda_{34}$} (d);
				\draw[edge] (d) to node[left] {$\lambda_{14}$} (a);
				
				\node[below left=0.1cm] at (a) {$v_1$};
				\node[below right=0.1cm] at (b) {$v_2$};
				\node[above right=0.1cm] at (c) {$v_3$};
				\node[above left=0.1cm] at (d) {$v_4$};
	    \end{tikzpicture}
	  \end{center}
	  \caption{Quadrilateral}
	  \label{fig:quadrilateral}
	\end{figure}
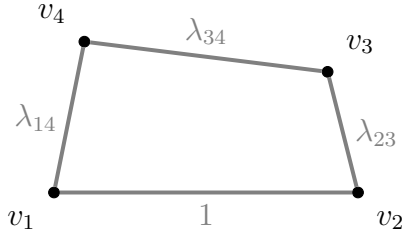

	Let $v_1v_2$ be the fixed edge. Let~$\lambda$ be a flexible labeling given by:
	\begin{align*}
		\lambda(v_1v_2)&= 1\,, & \lambda(v_2v_3)&= \lambda_{23}\,, &\lambda(v_3v_4)&= \lambda_{34}\,, & \lambda(v_4v_1)&= \lambda_{41}\,.
	\end{align*}
	The coordinates of~$v_i$ are denoted by $(x_i,y_i)$, so the system \eqref{eq:mainSystemOfEquations} has the form
	\begin{align*}
		&\begin{aligned}
			x_{1}=0, \quad	y_{1}&=0\,, \nonumber \\
			x_{2}=1, \quad y_{2}&=0\,,
		\end{aligned}
		& &\begin{aligned}
			{\left(x_{3} - 1\right)}^{2} + y_{3}^{2} &= \lambda_{23}^{2}\,, \\
			{\left(x_{3} - x_{4}\right)}^{2} + {\left(y_{3} - y_{4}\right)}^{2} &= \lambda_{34}^{2}\,, \\
			x_{4}^{2} + y_{4}^{2} &= \lambda_{41}^{2}\,. 
		\end{aligned}
	\end{align*}
	For generic~$\lambda_{23},\lambda_{34}$ and~$\lambda_{41}$, the solution set $X$ is an irreducible curve.
	Let $F$ be its complex function field. We have the following equations for $W_e,Z_e \in F$:
	\begin{align*}
		W_{12}= Z_{12} &=1\,, & W_{23} Z_{23} &= \lambda_{23}^{2}\,, & W_{34} Z_{34} &= \lambda_{34}^{2}\,, & W_{41} Z_{41} &= \lambda_{41}^{2}\,.
	\end{align*}
	Since~$G$ is a cycle, we also have
	\begin{align*}
		W_{23} + W_{34} + W_{41} + 1 &= 0 \,, & Z_{23} + Z_{34} + Z_{41} + 1 &= 0 \,.
	\end{align*}
	We choose the edge $e'$ to be~$v_1v_4$. Let~$\nu$ be a valuation such that~$\nu(W_{41})=1$ and~$\nu(c)=0$ for all $c\in\CC$.
	In order to compute valuations $W_{23}$ and $W_{34}$, we express $W_{34}$ in terms of the parameter~$u=W_{41}$, i.e., we solve the system
	\begin{align*}
		 \left(u + W_{34} + 1\right) \left(\frac{\lambda_{41}^{2}}{u} + Z_{34} + 1\right)&= \lambda_{23}^{2}\,,  & 	W_{34} Z_{34} &= \lambda_{34}^{2} \,.
	\end{align*}
	The solutions are
	\begin{equation*}
		W_{34}=\frac{-\lambda_{41}^{2} + (\lambda_{23}^{2} - \lambda_{34}^{2} - \lambda_{41}^{2} - 1) u - u^{2} \pm S}{2  (\lambda_{41}^{2} + u)}\,,
	\end{equation*}
	where
	\begin{equation*}
		S=\sqrt{
		\begin{aligned}
			&\lambda_{41}^{4}  + 2  (\lambda_{41}^{4} - (\lambda_{23}^{2} + \lambda_{34}^{2} - 1) \lambda_{41}^{2}) u \\
			&+((\lambda_{23}^{2}-1)^2 + \lambda_{34}^{4} + \lambda_{41}^{4} - 2  (\lambda_{23}^{2} + 1) \lambda_{34}^{2} - 2  (\lambda_{23}^{2} + \lambda_{34}^{2} - 2) \lambda_{41}^{2}) u^{2}\\
			 &- 2  (\lambda_{23}^{2} + \lambda_{34}^{2} - \lambda_{41}^{2} - 1) u^{3} + u^{4} 			
		\end{aligned}}.
	\end{equation*}
	The square root can be written as a power series by factoring out~$\lambda_{41}^{4}$ and 
	using the Taylor expansion of $\sqrt{1+t}=1 + \frac{t}{2}-\frac{t^2}{8}+O(t^3)$.
	Since the denominator of the solutions is a power series of order 0, it is invertible and hence the solutions are again power series, namely
	\begin{align*}
		W_{34}^1&=-1 + \frac{\lambda_{23}^{2} - \lambda_{41}^{2}}{\lambda_{41}^{2}}u + O(u^{2}) & &\text{and} 
		&	W_{34}^2&=-\frac{\lambda_{34}^{2}}{\lambda_{41}^{2}} u + O(u^{2})\,.
	\end{align*}
	So for the first solution, we have a NAC-coloring~$\delta_1$ given by:
	\begin{align*}
		\delta_1(v_1v_2)&= \blue{} &\impliedby \nu(W_{12})&= 0\,, &\delta_1(v_2v_3)&= \red{}\,,  \\
		\delta_1(v_3v_4)&= \blue{} &\impliedby \nu(W_{34}^1)&=0\,, & \delta_1(v_1v_4)&= \red{} &\impliedby \nu(W_{41})&=1\,.
	\end{align*}
	We remark that it is not necessary to compute $W_{23}$ precisely as we know that there are no almost \blue{} cycles, hence the edge~$v_2v_3$ must be \red{}.	
	Similarly, the second solution gives a NAC-coloring~$\delta_2$:
	\begin{align*}
		\delta_2(v_1v_2)&= \blue{} &\impliedby \nu(W_{12})&= 0\,, &	\delta_2(v_2v_3)&= \blue{}\,, \\
		\delta_2(v_3v_4)&= \red{} &\impliedby \nu(W_{34}^2)&=1\,,   &\delta_2(v_1v_4)&= \red{} &\impliedby \nu(W_{41})&=1\,.
	\end{align*}
	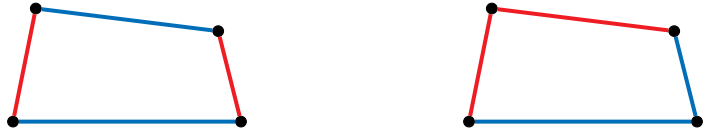
\begin{figure}[ht]
	  \begin{center}
	    \begin{tikzpicture}[scale=1.5]
	      \node[vertex] (a) at (0,0) {};
				\node[vertex] (b) at (2,0) {};
				\node[vertex] (c) at (1.8,0.8) {};
				\node[vertex] (d) at (0.2,1) {};
				
				\draw[bedge] (a) to (b);
				\draw[redge] (b) to (c);
				\draw[bedge] (c) to (d);
				\draw[redge] (d) to (a);
	    
				\begin{scope}[xshift=4cm]
					\node[vertex] (a) at (0,0) {};
					\node[vertex] (b) at (2,0) {};
					\node[vertex] (c) at (1.8,0.8) {};
					\node[vertex] (d) at (0.2,1) {};
					
					\draw[bedge] (a) to (b);
					\draw[bedge] (b) to (c);
					\draw[redge] (c) to (d);
					\draw[redge] (d) to (a);
				\end{scope}
	    \end{tikzpicture}
	  \end{center}
	  \caption{Quadrilateral with NAC-colorings}
	  \label{fig:quadrilateral2}
	\end{figure}
\end{exmp}

The following example illustrates the opposite implication of the proof of Theorem~\ref{thm:nacflexible}, i.e., how a flexible labeling is constructed from a NAC-coloring.
\begin{exmp}\label{ex:RotGraph}
	In Figure~\ref{fig:RotGraphCol}, a graph with its NAC-colorings and realizations compatible with the constructed flexible labelings are shown.
	Note that for the first two colorings we have overlapping vertices in the realizations.
	For the last coloring this does not happen. However, there are some non-visible edges.
	A generalization of $\rho_\alpha$ can prevent from this (see below).
	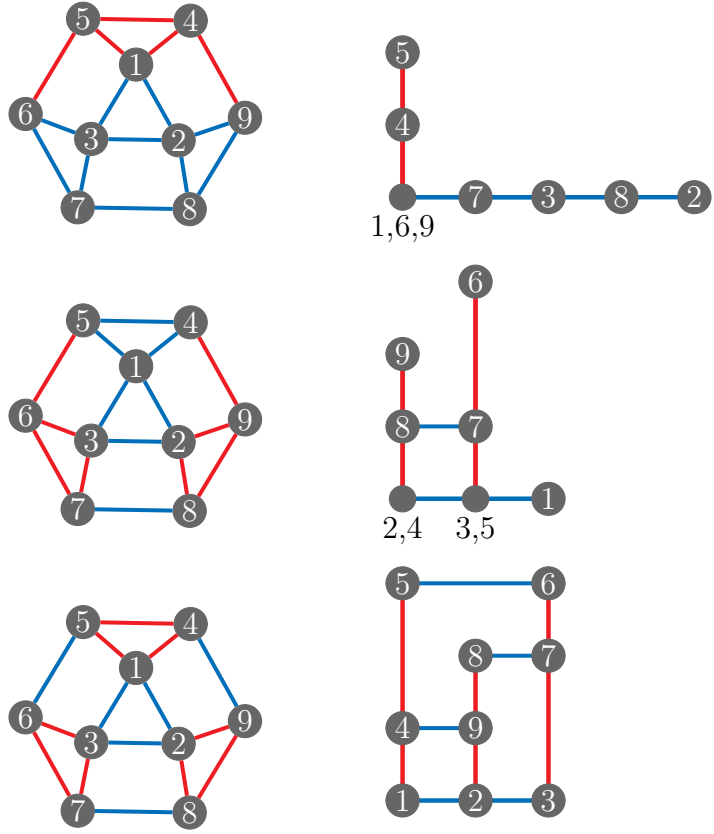
\begin{figure}[ht]
	\begin{center}
		\begin{tikzpicture}[scale=0.8]
			\begin{scope}
				\begin{scope}[scale=1.5]
					\node[lnode] (1) at (1.21142, 1.59773) {1};
					\node[lnode] (2) at (1.67716, 0.760731) {2};
					\node[lnode] (3) at (0.719916, 0.776553) {3};
					\node[lnode] (4) at (1.80994, 2.08308) {4};
					\node[lnode] (5) at (0.630919, 2.10474) {5};
					\node[lnode] (6) at (0., 1.05168) {6};
					\node[lnode] (7) at (0.571291, 0.0205644) {7};
					\node[lnode] (8) at (1.79883, 0.) {8};
					\node[lnode] (9) at (2.40542, 1.01018) {9};
					\draw[redge] (1)edge(4)  (5)edge(1) (4)edge(5) (5)edge(6) (9)edge(4);
					\draw[bedge] (1)edge(2) (2)edge(3) (3)edge(1)  (3)edge(6) (6)edge(7) (7)edge(3) (2)edge(8) (8)edge(9)	(9)edge(2)   (7)edge(8) ;
				\end{scope}
				\begin{scope}[xshift=5cm,yshift=-1cm,scale=1.2]
					\node[lnode] (1) at (1, 1) { };
					\node[lnode] (2) at (5, 1) { };
					\node[lnode] (3) at (3, 1) { };
					\node[lnode] (4) at (1, 2) { };
					\node[lnode] (5) at (1, 3) { };
					\node[lnode] (6) at (1, 1) { };
					\node[lnode] (7) at (2, 1) { };
					\node[lnode] (8) at (4, 1) { };
					\node[lnode] (9) at (1, 1) { };
					\draw[redge] (1)edge(4)  (5)edge(1) (4)edge(5) (5)edge(6) (9)edge(4);
					\draw[bedge] (1)edge(2) (2)edge(3) (3)edge(1)  (3)edge(6) (6)edge(7) (7)edge(3) (2)edge(8) (8)edge(9)	(9)edge(2)   (7)edge(8) ;
					\foreach \x in {2,3,4,5,7,8} \node[lnode] at (\x) {\x};
					\node[lnode] at (1) {\vphantom{9}};
					\node[below=3pt] at (1) {1,6,9};
				\end{scope}
			\end{scope}
			\begin{scope}[yshift=-5cm]
				\begin{scope}[scale=1.5]
					\node[lnode] (1) at (1.21142, 1.59773) {1};
					\node[lnode] (2) at (1.67716, 0.760731) {2};
					\node[lnode] (3) at (0.719916, 0.776553) {3};
					\node[lnode] (4) at (1.80994, 2.08308) {4};
					\node[lnode] (5) at (0.630919, 2.10474) {5};
					\node[lnode] (6) at (0., 1.05168) {6};
					\node[lnode] (7) at (0.571291, 0.0205644) {7};
					\node[lnode] (8) at (1.79883, 0.) {8};
					\node[lnode] (9) at (2.40542, 1.01018) {9};
					\draw[redge] (3)edge(6) (6)edge(7) (7)edge(3) (2)edge(8) (8)edge(9)	(9)edge(2) (5)edge(6) (9)edge(4);
					\draw[bedge] (1)edge(2) (2)edge(3) (3)edge(1) (1)edge(4) (4)edge(5) (5)edge(1) (7)edge(8);
				\end{scope}
				\begin{scope}[xshift=5cm,yshift=-1cm,scale=1.2]
					\node[lnode] (1) at (3, 1) {};
					\node[lnode] (2) at (1, 1) {};
					\node[lnode] (3) at (2, 1) {};
					\node[lnode] (4) at (1, 1) {};
					\node[lnode] (5) at (2, 1) {};
					\node[lnode] (6) at (2, 4) {};
					\node[lnode] (7) at (2, 2) {};
					\node[lnode] (8) at (1, 2) {};
					\node[lnode] (9) at (1, 3) {};
					\draw[redge] (3)edge(6) (6)edge(7) (7)edge(3) (2)edge(8) (8)edge(9)	(9)edge(2) (5)edge(6) (9)edge(4);
					\draw[bedge] (1)edge(2) (2)edge(3) (3)edge(1) (1)edge(4) (4)edge(5) (5)edge(1) (7)edge(8);
					\foreach \x in {1,6,7,8,9} \node[lnode] at (\x) {\x};
					\node[lnode] at (2) {\vphantom{2}};
					\node[below=3pt] at (2) {2,4};
					\node[lnode] at (3) {\vphantom{3}};
					\node[below=3pt] at (3) {3,5};
				\end{scope}
			\end{scope}			
			\begin{scope}[yshift=-10cm]
				\begin{scope}[scale=1.5]
					\node[lnode] (1) at (1.21142, 1.59773) {1};
					\node[lnode] (2) at (1.67716, 0.760731) {2};
					\node[lnode] (3) at (0.719916, 0.776553) {3};
					\node[lnode] (4) at (1.80994, 2.08308) {4};
					\node[lnode] (5) at (0.630919, 2.10474) {5};
					\node[lnode] (6) at (0., 1.05168) {6};
					\node[lnode] (7) at (0.571291, 0.0205644) {7};
					\node[lnode] (8) at (1.79883, 0.) {8};
					\node[lnode] (9) at (2.40542, 1.01018) {9};
					\draw[redge] (1)edge(4) (4)edge(5) (5)edge(1) (3)edge(6) (6)edge(7) (7)edge(3) (2)edge(8) (8)edge(9)	(9)edge(2);
					\draw[bedge] (1)edge(2) (2)edge(3) (3)edge(1) (9)edge(4) (5)edge(6) (7)edge(8);
				\end{scope}
				\begin{scope}[xshift=5cm,yshift=-1cm,scale=1.2]
					\node[lnode] (1) at (1, 1) {};
					\node[lnode] (2) at (2, 1) {};
					\node[lnode] (3) at (3, 1) {};
					\node[lnode] (4) at (1, 2) {};
					\node[lnode] (5) at (1, 4) {};
					\node[lnode] (6) at (3, 4) {};
					\node[lnode] (7) at (3, 3) {};
					\node[lnode] (8) at (2, 3) {};
					\node[lnode] (9) at (2, 2) {};
					\draw[redge] (1)edge(4) (4)edge(5) (5)edge(1) (3)edge(6) (6)edge(7) (7)edge(3) (2)edge(8) (8)edge(9)	(9)edge(2);
					\draw[bedge] (1)edge(2) (2)edge(3) (3)edge(1) (9)edge(4) (5)edge(6) (7)edge(8);
					\foreach \x in {1,...,9} \node[lnode] at (\x) {\x};
				\end{scope}
			\end{scope}
		\end{tikzpicture}
	\end{center}
	\caption{For each NAC-coloring (left side), the proof of the implication $\impliedby$ of Theorem~\ref{thm:nacflexible} 
                 gives a flexible labeling together with a curve of realizations. These realizations (right side) are often
                 ``degenerate'' in the sense that vertex points may coincide or that edge line segments may overlap.}
	\label{fig:RotGraphCol}
	\end{figure}
\end{exmp}

In order to have more flexibility in the choice of lengths of a flexible labeling,
we modify the ``$\impliedby$'' part of the proof of Theorem~\ref{thm:nacflexible} by using a ``zigzag'' grid instead of the normal one.
\begin{center}
	\begin{tikzpicture}[scale=0.9]
		\begin{scope}[xshift=0cm]
			\foreach \a in  {(0,0),(0.2,1.1),(-0.1,1.9),(0.2,3)}
			{
				\draw[gridl] \a
				\foreach \b in {(0,0),(1,0.2),(2.1,-0.1),(3.3,0.1),(3.9,-0.2)}
				{
					-- +\b
				};
			}
			\foreach \b in {(0,0),(1,0.2),(2.1,-0.1),(3.3,0.1),(3.9,-0.2)}
			{
				\draw[gridl] \b node[gridp] {}
				\foreach \a in {(0,0),(0.2,1.1),(-0.1,1.9),(0.2,3)}
				{
					-- +\a node[gridp] {}
				};
			}
		\end{scope}
		\begin{scope}[xshift=6cm]
			\foreach \x/\y in  {0/0,0.2/1.1,-0.1/1.9,0.2/3}
			{
				\begin{scope}[rotate around={-20:(0,0)}]
					\coordinate (ar) at (\x,\y);
				\end{scope}				
				\draw[gridl] (ar)
				\foreach \b in {(0,0),(1,0.2),(2.1,-0.1),(3.3,0.1),(3.9,-0.2)}
				{
					-- +\b
				};
			}
			\foreach \b in {(0,0),(1,0.2),(2.1,-0.1),(3.3,0.1),(3.9,-0.2)}
			{
				\draw[gridl,rotate around={-20:\b}] \b node[gridp] {}
				\foreach \a in {(0,0),(0.2,1.1),(-0.1,1.9),(0.2,3)}
				{
					-- +\a node[gridp] {}
				};
			}
		\end{scope}
	\end{tikzpicture}
\end{center}
Let $R_1, \dots, R_m$ and $B_1, \dots, B_n$ be as before.
Let $a_1, \dots, a_n$ and $b_1, \dots , b_m$ be pairwise distinct vectors in $\RR^2$.
For $\alpha\in [0,2\pi)$, we define a map $\rho'_\alpha\colon V_G\rightarrow \RR^2$ by 
\begin{equation*}
  \rho'_\alpha(v)=
  \begin{pmatrix}
		\cos \alpha & \sin\alpha \\
		-\sin\alpha & \cos \alpha
	\end{pmatrix}%
	\cdot a_j +b_i \,,
\end{equation*}
where $i$ and $j$ are such that~$v\in R_i \cap B_j$.
The map $\rho'_\alpha$ gives a labeling $\lambda'_\alpha\colon E_G\rightarrow \RR_+$, where
\begin{equation*}
  \lambda'_\alpha(uv)=\norm{\rho'_\alpha(u)-\rho'_\alpha(v)}=\norm{%
  \begin{pmatrix}
		\cos \alpha & \sin\alpha \\
		-\sin\alpha & \cos \alpha
	\end{pmatrix}%
	\cdot(a_j-a_l) + (b_i-b_k)}
\end{equation*}
for all edges~$uv\in E_G$ and $i,j,k,l$ such that~$u\in R_i\cap B_j$ and~$v\in R_k\cap B_l$. 
If~$uv$ is \red{}, then $i=k$ and we have 
\begin{equation*}
  \lambda'_\alpha(uv)=\norm{%
  \begin{pmatrix}
		\cos \alpha & \sin\alpha \\
		-\sin\alpha & \cos \alpha
	\end{pmatrix}%
	\cdot(a_j-a_l)}=\norm{a_j-a_l}.
\end{equation*}
If~$uv$ is \blue{}, then $j=l$ and hence
\begin{equation*}
  \lambda'_\alpha(uv)=\norm{(b_i-b_k)}.
\end{equation*}
Therefore,~$\lambda'_\alpha$ is independent on $\alpha$, let $\lambda'=\lambda'_\frac{\pi}{2}$.
Moreover, the codomain of~$\lambda'$ is indeed $\RR_+$: if $\lambda'(uv)=0$ for some edge~$uv\in E_G$,
then both~$u$ and~$v$ belong to the same $R_i\cap B_j$, which leads to a contradiction in the same way as in the previous proof.
To conclude, the labeling~$\lambda'$ is flexible since there are infinitely many realizations $\rho'_\alpha$
which are non equivalent as~$\delta$ is surjective.

We apply this construction to the previous example.
\begin{exmp}
	We consider the last coloring from Example~\ref{ex:RotGraph}; we keep the order of equivalence classes and fix
	\begin{align*}
	  a_1 &= (0, 0)\,, & 
	  a_2 &= \left(-\frac{3 \sqrt{3}}{8},-\frac{3}{8}\right), &
	  a_3 &= \left(0,\frac{3}{4}\right), &
	  a_4 &= \left(\frac{3 \sqrt{3}}{8},-\frac{3}{8}\right),\\
	  b_1 &= (0,0)\,, &
	  b_2 &= \left(-\frac{1}{2},\frac{\sqrt{3}}{2}\right), &
	  b_3 &= \left(\frac{1}{2},\frac{\sqrt{3}}{2}\right).
	\end{align*}
	Then we get the following flexible realization.
	The flexibility comes from rotating the outer triangles around the vertices of the inner triangle respectively.
	\begin{center}
		\begin{tikzpicture}[scale=1.5]
			\node[lnode] (1) at (0., 0.) {1};
			\node[lnode] (2) at (-0.5, 0.866025) {2};
			\node[lnode] (3) at (0.5, 0.866025) {3};
			\node[lnode] (4) at (-0.649519, -0.375) {4};
			\node[lnode] (5) at (0.649519, -0.375) {5};
			\node[lnode] (6) at (1.14952, 0.491025) {6};
			\node[lnode] (7) at (0.5, 1.61603) {7};
			\node[lnode] (8) at (-0.5, 1.61603) {8};
			\node[lnode] (9) at (-1.14952, 0.491025) {9};
			\draw[redge] (1)edge(4) (4)edge(5) (5)edge(1) (3)edge(6) (6)edge(7) (7)edge(3) (2)edge(8) (8)edge(9)	(9)edge(2);
			\draw[bedge] (1)edge(2) (2)edge(3) (3)edge(1) (9)edge(4) (5)edge(6) (7)edge(8);
		\end{tikzpicture}
	\end{center}
	\begin{center}
	  \begin{tikzpicture}[scale=0.57]
			\begin{scope}[]
				\node[cnode] (1) at (0., 0.) {};
				\node[cnode] (2) at (-0.5, 0.866025) {};
				\node[cnode] (3) at (0.5, 0.866025) {};
				\node[cnode] (4) at (-0.649519, -0.375) {};
				\node[cnode] (5) at (0.649519, -0.375) {};
				\node[cnode] (6) at (1.14952, 0.491025) {};
				\node[cnode] (7) at (0.5, 1.61603) {};
				\node[cnode] (8) at (-0.5, 1.61603) {};
				\node[cnode] (9) at (-1.14952, 0.491025) {};
				\draw[redge] (1)edge(4) (4)edge(5) (5)edge(1) (3)edge(6) (6)edge(7) (7)edge(3) (2)edge(8) (8)edge(9)	(9)edge(2);
				\draw[bedge] (1)edge(2) (2)edge(3) (3)edge(1) (9)edge(4) (5)edge(6) (7)edge(8);
			\end{scope}
			\begin{scope}[xshift=3cm]
				\node[cnode] (1) at (0., 0.) {};
				\node[cnode] (2) at (-0.5, 0.866025) {};
				\node[cnode] (3) at (0.5, 0.866025) {};
				\node[cnode] (4) at (-0.724444, -0.194114) {};
				\node[cnode] (5) at (0.53033, -0.53033) {};
				\node[cnode] (6) at (1.03033, 0.335695) {};
				\node[cnode] (7) at (0.694114, 1.59047) {};
				\node[cnode] (8) at (-0.305886, 1.59047) {};
				\node[cnode] (9) at (-1.22444, 0.671911) {};
				\draw[redge] (1)edge(4) (4)edge(5) (5)edge(1) (3)edge(6) (6)edge(7) (7)edge(3) (2)edge(8) (8)edge(9)	(9)edge(2);
				\draw[bedge] (1)edge(2) (2)edge(3) (3)edge(1) (9)edge(4) (5)edge(6) (7)edge(8);
			\end{scope}
			\begin{scope}[xshift=6cm]
				\node[cnode] (1) at (0., 0.) {};
				\node[cnode] (2) at (-0.5, 0.866025) {};
				\node[cnode] (3) at (0.5, 0.866025) {};
				\node[cnode] (4) at (-0.75, 0.) {};
				\node[cnode] (5) at (0.375, -0.649519) {};
				\node[cnode] (6) at (0.875, 0.216506) {};
				\node[cnode] (7) at (0.875, 1.51554) {};
				\node[cnode] (8) at (-0.125, 1.51554) {};
				\node[cnode] (9) at (-1.25, 0.866025) {};
				\draw[redge] (1)edge(4) (4)edge(5) (5)edge(1) (3)edge(6) (6)edge(7) (7)edge(3) (2)edge(8) (8)edge(9)	(9)edge(2);
				\draw[bedge] (1)edge(2) (2)edge(3) (3)edge(1) (9)edge(4) (5)edge(6) (7)edge(8);
			\end{scope}
			\begin{scope}[xshift=9cm]
				\node[cnode] (1) at (0., 0.) {};
				\node[cnode] (2) at (-0.5, 0.866025) {};
				\node[cnode] (3) at (0.5, 0.866025) {};
				\node[cnode] (4) at (-0.724444, 0.194114) {};
				\node[cnode] (5) at (0.194114, -0.724444) {};
				\node[cnode] (6) at (0.694114, 0.141581) {};
				\node[cnode] (7) at (1.03033, 1.39636) {};
				\node[cnode] (8) at (0.0303301, 1.39636) {};
				\node[cnode] (9) at (-1.22444, 1.06014) {};
				\draw[redge] (1)edge(4) (4)edge(5) (5)edge(1) (3)edge(6) (6)edge(7) (7)edge(3) (2)edge(8) (8)edge(9)	(9)edge(2);
				\draw[bedge] (1)edge(2) (2)edge(3) (3)edge(1) (9)edge(4) (5)edge(6) (7)edge(8);
			\end{scope}
			\begin{scope}[xshift=12cm]
				\node[cnode] (1) at (0., 0.) {};
				\node[cnode] (2) at (-0.5, 0.866025) {};
				\node[cnode] (3) at (0.5, 0.866025) {};
				\node[cnode] (4) at (-0.649519, 0.375) {};
				\node[cnode] (5) at (0., -0.75) {};
				\node[cnode] (6) at (0.5, 0.116025) {};
				\node[cnode] (7) at (1.14952, 1.24103) {};
				\node[cnode] (8) at (0.149519, 1.24103) {};
				\node[cnode] (9) at (-1.14952, 1.24103) {};
				\draw[redge] (1)edge(4) (4)edge(5) (5)edge(1) (3)edge(6) (6)edge(7) (7)edge(3) (2)edge(8) (8)edge(9)	(9)edge(2);
				\draw[bedge] (1)edge(2) (2)edge(3) (3)edge(1) (9)edge(4) (5)edge(6) (7)edge(8);
			\end{scope}
			\begin{scope}[xshift=15cm]
				\node[cnode] (1) at (0., 0.) {};
				\node[cnode] (2) at (-0.5, 0.866025) {};
				\node[cnode] (3) at (0.5, 0.866025) {};
				\node[cnode] (4) at (-0.53033, 0.53033) {};
				\node[cnode] (5) at (-0.194114, -0.724444) {};
				\node[cnode] (6) at (0.305886, 0.141581) {};
				\node[cnode] (7) at (1.22444, 1.06014) {};
				\node[cnode] (8) at (0.224444, 1.06014) {};
				\node[cnode] (9) at (-1.03033, 1.39636) {};
				\draw[redge] (1)edge(4) (4)edge(5) (5)edge(1) (3)edge(6) (6)edge(7) (7)edge(3) (2)edge(8) (8)edge(9)	(9)edge(2);
				\draw[bedge] (1)edge(2) (2)edge(3) (3)edge(1) (9)edge(4) (5)edge(6) (7)edge(8);
			\end{scope}
			\begin{scope}[xshift=18cm]
				\node[cnode] (1) at (0., 0.) {};
				\node[cnode] (2) at (-0.5, 0.866025) {};
				\node[cnode] (3) at (0.5, 0.866025) {};
				\node[cnode] (4) at (-0.375, 0.649519) {};
				\node[cnode] (5) at (-0.375, -0.649519) {};
				\node[cnode] (6) at (0.125, 0.216506) {};
				\node[cnode] (7) at (1.25, 0.866025) {};
				\node[cnode] (8) at (0.25, 0.866025) {};
				\node[cnode] (9) at (-0.875, 1.51554) {};
				\draw[redge] (1)edge(4) (4)edge(5) (5)edge(1) (3)edge(6) (6)edge(7) (7)edge(3) (2)edge(8) (8)edge(9)	(9)edge(2);
				\draw[bedge] (1)edge(2) (2)edge(3) (3)edge(1) (9)edge(4) (5)edge(6) (7)edge(8);
			\end{scope}
	  \end{tikzpicture}
	\end{center}
\end{exmp}

To conclude this section, we discuss whether all NAC-colorings of a graph are obtained
from some curve of realizations compatible with a flexible labeling, 
or if all flexible labelings of a graph can be expressed in terms of ``zigzag'' grid construction.

We recall that all NAC-colorings of a quadrilateral in which the edge~$v_1v_2$ is \blue{} were obtained in Example~\ref{ex:quadrilateral}.
In general, a single flexible labeling does not give all possible NAC-colorings. But every NAC-coloring comes from some flexible labeling.
Namely, we show that if we start with a NAC-coloring~$\delta$ of~$G$ and apply the construction $\impliedby$ to get a flexible labeling $\lambda$,
then the construction $\implies$ applied on~$\lambda$ gives back the NAC-coloring~$\delta$.

Let $\lambda$ be a labeling, $\rho_\alpha$ some realizations and we use notation as in the $\impliedby$ part of the proof of Theorem~\ref{thm:nacflexible}.
The realizations $\rho_\alpha$ for $\alpha\in [0,2\pi)$ give a parametrization of an irreducible curve $C$ of realizations compatible with~$\lambda$.
If~$uv\in E_G$ is an edge such that $\delta(uv)=\red{}$, then there are indices $j,k,l\in \NN$ such that~$u\in R_j \cap B_k$ and~$v\in R_j \cap B_{k+l}$.
Hence
\begin{align*}
	W_{uv}&=j+k\cos\alpha - (j+(k+l)\cos\alpha) + \ci(k\sin\alpha-(k+l)\sin\alpha )\\
	&=l(\cos\alpha +\ci \sin\alpha)\,.
\end{align*}
Therefore, $z= \cos\alpha +\ci \sin\alpha$ is a transcendental element of the complex function field of $C$
and there is a valuation~$\nu$ such that~$\nu(z)=1$ and $\nu(\CC)=\{0\}$. Thus, $\nu(W_e)=1$ for all $e\in E_G$ such that $\delta(e)=\red{}$.
On the other hand, if~$u'v'\in E_G$ is such that $\delta(u'v')=\blue{}$,
that means~$u'\in R_{j'} \cap B_{k'}$ and~$v'\in R_{j'+l'} \cap B_{k'}$ for some $j',k',l'\in \NN$, then
\begin{align*}
	W_{u'v'}&=j'+k'\cos\alpha - ((j'+l')+ k'\cos\alpha) + \ci(k'\sin\alpha-k'\sin\alpha)=-l'\,.
\end{align*}
Therefore, $\nu(W_{e'})=0$ for all $e'\in E_G$ such that $\delta(e')=\blue{}$.
We showed that the construction from a flexible labeling~$\lambda$ gives the coloring~$\delta$,
since the valuation of \red{} edges is positive and the valuation of $\blue{}$ edges is zero.

While every NAC-coloring comes from a flexible labeling, there are flexible labelings which cannot be obtained by a ``zigzag'' grid.
This can be seen already on the quadrilateral -- the described construction always assigns same lengths to the opposite edges, 
but there are also flexible labelings with different lengths.
Another example is~$K_{2,3}$. A generic labeling of~$K_{2,3}$ is flexible and the realizations are injective,
but at least two vertices coincide in the grid construction from any possible NAC-coloring (see Figure~\ref{fig:K23}).
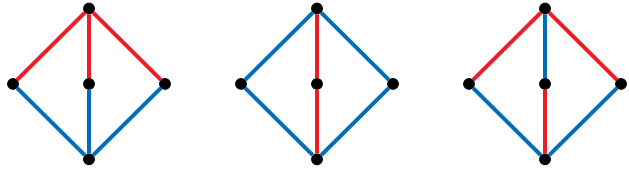
\begin{figure}[ht]
  \begin{center}
    \begin{tikzpicture}
      \begin{scope}
        \node[vertex] (a) at (0,0) {};
				\node[vertex] (b) at (0,1) {};
				\node[vertex] (c) at (0,-1) {};
				\node[vertex] (d) at (1,0) {};
				\node[vertex] (e) at (-1,0) {};
				
				\draw[redge] (a)edge(b) (b)edge(d) (b)edge(e);
				\draw[bedge] (a)edge(c)  (c)edge(d) (c)edge(e);
      \end{scope}
      \begin{scope}[xshift=3cm]
        \node[vertex] (a) at (0,0) {};
				\node[vertex] (b) at (0,1) {};
				\node[vertex] (c) at (0,-1) {};
				\node[vertex] (d) at (1,0) {};
				\node[vertex] (e) at (-1,0) {};
				
				\draw[redge] (a)edge(b) (a)edge(c);
				\draw[bedge] (b)edge(d) (b)edge(e) (c)edge(d) (c)edge(e);
      \end{scope}
      \begin{scope}[xshift=6cm]
        \node[vertex] (a) at (0,0) {};
				\node[vertex] (b) at (0,1) {};
				\node[vertex] (c) at (0,-1) {};
				\node[vertex] (d) at (1,0) {};
				\node[vertex] (e) at (-1,0) {};
				
				\draw[redge] (a)edge(c) (b)edge(d) (b)edge(e);
				\draw[bedge] (a)edge(b)  (c)edge(d) (c)edge(e);
      \end{scope}
    \end{tikzpicture}
  \end{center}
  \caption{NAC-colorings of $K_{2,3}$}
  \label{fig:K23}
\end{figure}
On the other hand, if there is a NAC-coloring such that $|B_i\cap R_j|\leq 1$ for all $i,j$,
then there is a curve of injective realizations compatible with the obtained flexible labeling.

\section{Classes of Graphs with and without Flexible Labelings} \label{sec:examples}

In this section, we give two sufficient conditions for the non-existence of a flexible labeling of a given graph, as well
as three sufficient conditions for the existence of a flexible labeling of a given graph.
We also give a sharp upper bound for the number of edges of a graph on $n$ vertices with a flexible labeling.

The triangle graph $C_3$ is always rigid and it is natural to expect that subgraphs which ``consist of triangles'' are also rigid.
The notion of \trcon{}-ness makes this intuition precise.
\begin{defn}
	Let~$G$ be a graph.
	Let $\trequiv'$ be a relation on $E_G\times E_G$ such that $e_1 \trequiv' e_2$
	iff there exists a triangle subgraph $C_3$ of~$G$ such that $e_1, e_2\in E_{C_3}$.
	Let $\trequiv$ be the reflexive-transitive closure of $\trequiv'$.
	The graph~$G$ is called \emph{\trcon{}} if $e_1\trequiv e_2$ for all $e_1,e_2\in E_G$.
	An edge~$e\in E_G$ is called a \emph{connecting edge} if it belongs to no triangle subgraph.
\end{defn}
Note that the graph consisting of a single edge is therefore \trcon{}.
Now we can justify the fact that \trcon{} subgraphs of a graph are indeed rigid, since their edges are colored by the same color.
\begin{lem}	\label{lem:coloringOfEquivalentEdges}
	Let~$\delta$ be a coloring of a graph~$G$ such that there are no almost \red{} neither almost \blue{} cycles.
	If $H$ is a \trcon{} subgraph of~$G$, then $\delta(e)=\delta(e')$ for all $e,e'\in E_H$.
\end{lem}
\begin{proof}
	The claim follows from the fact that $\trequiv$ is the reflexive-transitive closure of $\trequiv'$ and all edges of a triangle must have the same color.
\qed\end{proof}

An immediate consequence is a necessary condition for the existence of a NAC-coloring.
\begin{thm}\label{thm:necessaryCondNACcoloring}
	If a graph~$G$ has a NAC-coloring, then there is no subgraph~$H$ of~$G$ such that $V_G=V_H$ and $H$ is \trcon.
\end{thm}
\begin{proof}
	Assume that there is such a subgraph $H$. By Lemma~\ref{lem:coloringOfEquivalentEdges},
	all edges of $H$ must be colored by the same color, let us say \red{}.
	Since $H$ spans the whole~$G$, also the edges in~$E_G\setminus E_H$ must be \red{} according to Lemma~\ref{lem:equivCycleCond}.
	This contradicts the surjectivity of the NAC-coloring.
\qed\end{proof}
We remark that the fact that a graph is not spanned by a \trcon{} subgraph is a necessary condition, not sufficient.
Some counterexamples, i.e., graphs without a NAC-coloring which are not spanned by a \trcon{} subgraph, are shown in Figure~\ref{fig:countergeneral}.
In Section~\ref{sec:conjecture} we discuss a class of graphs for which this might indeed be a sufficient condition.
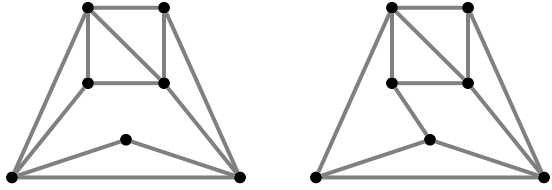
\begin{figure}[ht]
  \begin{center}
    \begin{tikzpicture}
      \begin{scope}[]
        \node[vertex] (a) at (-0.5,-0.75) {};
				\node[vertex] (b) at (0.5,0.5) {};
				\node[vertex] (c) at (1.5,0.5) {};
				\node[vertex] (d) at (2.5,-0.75) {};
				\node[vertex] (e) at (0.5,1.5) {};
				\node[vertex] (f) at (1.5,1.5) {};
				\node[vertex] (g) at (1,-0.25) {};
				\draw[edge] (a)edge(b) (a)edge(d) (a)edge(e) (b)edge(c) (b)edge(e) (c)edge(d) (c)edge(e) (c)edge(f) (d)edge(f) (e)edge(f) (g)edge(a) (g)edge(d);
      \end{scope}
      \begin{scope}[xshift=4cm]
        \node[vertex] (a) at (-0.5,-0.75) {};
				\node[vertex] (b) at (0.5,0.5) {};
				\node[vertex] (c) at (1.5,0.5) {};
				\node[vertex] (d) at (2.5,-0.75) {};
				\node[vertex] (e) at (0.5,1.5) {};
				\node[vertex] (f) at (1.5,1.5) {};
				\node[vertex] (g) at (1,-0.25) {};
				\draw[edge]  (a)edge(d) (a)edge(e) (b)edge(c) (b)edge(e) (c)edge(d) (c)edge(e) (c)edge(f) (d)edge(f) (e)edge(f) (g)edge(a) (g)edge(b) (g)edge(d);
      \end{scope}      
    \end{tikzpicture}
  \end{center}
  \caption{Graphs which are not \trcon\ and have no NAC-coloring}
  \label{fig:countergeneral}
\end{figure}

We focus now on some sufficient conditions for the existence of a NAC-coloring.
Obviously, if a graph is disconnected in a way that at least two components contain an edge,
then a NAC coloring can be obtained by coloring all edges in one of these components to \red{} and the rest to \blue{}.
The coloring is surjective and there might be only \red{} cycles or \blue{} cycles.

The following two theorems and corollary construct a NAC-coloring for connected graphs satisfying certain sufficient conditions.
\begin{thm}\label{thm:independentVertexCut}
	Let~$G$ be a connected graph. If there exists an independent set of vertices $V_{c}$ which separates~$G$, then~$G$ has a NAC-coloring.
\end{thm}
\begin{proof}
	We recall that a set of vertices is independent, if no two vertices of it are connected by an edge.
	The assumption that $V_c$ separates $G$	means that the subgraph $H$ of $G$ induced by $V\setminus V_c$ is disconnected.
	Let $M$ be a union of the connected components of $H$ such that $M\neq H$. Set
	\begin{equation*}
	  \delta(e)=
	  \begin{cases} 
			\red{} & \text{if } e\cap V_M \neq \emptyset \\
			\blue{} & \text{otherwise,}
		\end{cases}
	\end{equation*}
	for all $e \in E_G$. Now if~$uv, u'v \in E_G$ are edges such that $\delta(uv)\neq\delta(u'v)$, then~$v\in V_c$.
	Hence, there is no almost \red{} cycle or almost \blue{} cycle since vertices of $V_c$ are nonadjacent by assumption.
\qed\end{proof}

\begin{cor}\label{cor:vertexNotInTriangle}
	Let~$G$ be a connected graph such that $|E_G|\geq 2$. If there is a vertex~$v \in V_G$ such that it is not contained in any triangle $C_3\subset G$,
	then the graph~$G$ has a NAC-coloring.
\end{cor}
\begin{proof}
	If $G$ is a star graph, then it has clearly a NAC-coloring.
	Otherwise, the neighbors of~$v$ separate $G$ and are nonadjacent by assumption.
	Thus, there is a NAC-coloring by Theorem~\ref{thm:independentVertexCut}.
\qed\end{proof}

\begin{thm}\label{thm:connectingEdgesCut}
	Let~$G$ be a connected graph, $|E_G|\geq 2$.
	If there is a set $E_c$ of connecting edges of~$G$ such that
	$E_c$ separates~$G$ and the subgraph of~$G$ induced by $E_c$ contains no path of length four,
	then~$G$ has a NAC-coloring.
\end{thm}
\begin{proof}
	Let $E'_c$ be a minimal subset of $E_c$ which separates~$G$, i.e.,
	$E'_c$ is a minimal subset of $E_c$ such that $G'=(V_G, E_G\setminus E'_c)$ is disconnected.
	Set
	\begin{equation*}
	  \delta(e)=
	  \begin{cases}
			\red{} & \text{if } e\in E'_c \\
			\blue{} & \text{otherwise,}
		\end{cases}
	\end{equation*}
	for all $e \in E_G$.
	The existence of an almost \blue{} cycle with a \red{} edge~$u_rv_r$ contradicts the minimality of $E'_c$:
	since~$u_r$ and~$v_r$ are connected by a \blue{} path, they belong to the same connected component in~$G'$.
	Hence, $E'_c\setminus\{u_ru_v\}$ separates~$G$. 

	On the other hand, assume that there exists an almost \red{} cycle $C$.
	Since there is no path of length four in the subgraph of~$G$ induced by $E_c$, the length of $C$ is three or four.
	The length three contradicts the assumption that the \red{} edges are in~$E_c$,
	i.e., they are connecting edges.
	If $C$ has length four and the \red{} edges are $ u_1u_2, u_2u_3, u_3u_4$,
	then $G'$ has at least three connected components, since there are no almost \blue{} cycles,
	i.e., $u_1$ and~$u_4$ are in a different component than~$u_2$ and~$u_3$, and~$u_2$ and~$u_3$ are not in the same component.
	But the minimality of $E'_c$ implies that $G'$ has only two connected components which is a contradiction.

	The surjectivity of~$\delta$ follows from minimality of $E_c$ and the assumption $|E_G|\geq 2$.
\qed\end{proof}

A graph which does not satisfy the assumption of Theorem~\ref{thm:independentVertexCut} neither Theorem~\ref{thm:connectingEdgesCut},
but still has a NAC-coloring is shown in Figure~\ref{fig:noSeparatingSets}.
\begin{figure}[ht]
  \begin{center}
    \begin{tikzpicture}
      \begin{scope}
        \node[vertex] (a) at (0,0) {};
		\node[vertex] (b) at (2,0) {};
		\node[vertex] (c) at (1,1) {};
		\node[vertex] (d) at (0,3) {};
		\node[vertex] (e) at (2,3) {};
		\node[vertex] (f) at (1,2) {};
		\node[vertex] (g) at (3,1.5) {};
		
		\draw[redge] (a)edge(b) (b)edge(c) (a)edge(c) (d)edge(e) (d)edge(f) (e)edge(f);
		\draw[bedge] (a)edge(d) (c)edge(f) (b)edge(e) (b)edge(g) (g)edge(e);
      \end{scope}
    \end{tikzpicture}
  \end{center}
  \caption{A graph with a NAC-coloring which has no independent separating set of vertices or separating set of connecting edges.}
  \label{fig:noSeparatingSets}
\end{figure}
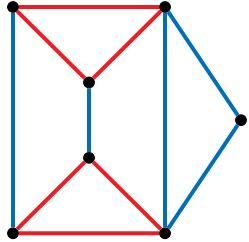

In Figure~\ref{fig:manyedges} we have shown two graphs with $\frac{n(n-1)}{2}-(n-2)$ edges.
Both have a flexible labeling according to Corollary~\ref{cor:vertexNotInTriangle}.
We show now that this is the maximal number of edges that allows a flexible labeling.
\begin{thm}
	Let~$G$ be a graph, $n=|V_G|$. If~$G$ has a flexible labeling, then $|E_G|\leq \frac{n(n-1)}{2}-(n-2)$.
\end{thm}
\begin{proof}
	We can assume that~$G$ is connected, since a disconnected graph has at most $\frac{n(n-1)}{2}-(n-1)$ edges.
	If a graph has a flexible labeling, then it is not \trcon{} by Theorem~\ref{thm:necessaryCondNACcoloring}.
	Hence, there are two adjacent edges~$vu_1$ and~$vu_2$ such that~$vu_1 \nottrequiv vu_2$. Thus, $u_1u_2\notin E_G$. Let $u_3,\ldots,u_k$ be the other neighbors of~$v$.
	For all $i\in \{3,\dots, k\}$, we have~$u_iu_1\notin E_G$ or~$u_iu_2\notin E_G$, otherwise~$vu_1 \trequiv vu_2$.
	Since the complete graph on $n$ vertices has $\frac{n(n-1)}{2}$ edges and $n-1-k$ vertices are not neighbors of~$v$, we have
	\begin{equation*}
		|E_G|\leq \frac{n(n-1)}{2} - 1 - (k-2) - (n-1-k)=\frac{n(n-1)}{2} - (n-2)\,.
	\end{equation*}
\qed\end{proof}

\section{A Conjecture on Laman graphs}\label{sec:conjecture}

In this section, we give a a conjectural characterization of the Laman graphs that have a flexible
labeling. By a computer search, we verified that conjecture for all Laman graphs with at most 12
vertices. In order to prove the conjecture, it would suffice to prove it for ``problematic cases'',
to be defined in Definition~\ref{def:problematicGraphs}.

\begin{conj}
	Let~$G$ be a Laman graph. If~$G$ is not \trcon{}, then it has a NAC-coloring, i.e., it has a flexible labeling.
\end{conj}
To support the conjecture, we prove it for a large class of Laman graphs in Theorem~\ref{thm:lamanHasNACorProblematic}. 

For the proof, we need the fact that every Laman graph~$G$ can be constructed by so called Henneberg moves (see for instance \cite{Laman1970}).
This means that there exists a sequence of graphs $G_0,G_1, \dots, G_n=G$ such that $G_0$ is a single edge
and $G_i$ is obtained from $G_{i-1}$ by Henneberg move I or II according to the following definition (see also Figure~\ref{fig:henneberg}).

\begin{defn}\label{def:typesOfHennebergMoves}
	Let $G$ and $G'$ be Laman graphs such that $V_G=\{v\}\cup V_G$, where~$v\notin V_{G'}$. The graph~$G$ is constructed from~$G'$ 
	\begin{enumerate}
		\item by a \emph{Henneberg move I} from vertices~$u,u'\in V_{G'}$ if $E_G=E_{G'}\cup\{uv, u'v\}$.
		If~$uu'\in E$, it is a \emph{Henneberg move of type Ia}, otherwise a \emph{Henneberg move of type Ib}.
		\item by a \emph{Henneberg move of type II} from vertices $u,u_1,u_2 \in V'$ if $u_1u_2\in E'$ and $E_G=\{uv,u_1v,u_2v\}\cup E_{G'}\setminus \{u_1u_2\}$.
		If~$uu_1\notin E$ and~$uu_2\notin E$, it is a \emph{Henneberg move of type IIa}.
		If either~$uu_1\in E$ or~$uu_2\in E$, it is a \emph{Henneberg move of type IIb}.
		If~$uu_1\in E$ and~$uu_2\in E$, it is a \emph{Henneberg move of type~IIc}.
	\end{enumerate}
\end{defn}
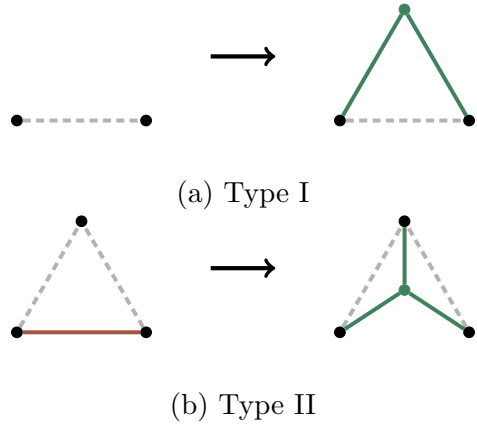
\begin{figure}[ht]
  \begin{center}
    \begin{subfigure}[b]{0.98\textwidth}
			\begin{center}
			  \begin{tikzpicture}[scale=1.7]
					\node[vertex] (a) at (0,0) {};
					\node[vertex] (b) at (1,0) {};
					\node[vertex] (d) at (2.5,0) {};
					\node[vertex] (e) at (3.5,0) {};
					\node[nvertex] (f) at (3,0.866025) {};
					\draw[edgeq] (a)edge(b) (d)edge(e);
					\draw[nedge] (d)edge(f) (e)edge(f);
					\draw[ultra thick,->] (1.5,0.5) -- (2,0.5);
				\end{tikzpicture}%
			\end{center}    
			\caption{Type~I}\label{fig:type1}
		\end{subfigure}
		\begin{subfigure}[b]{0.98\textwidth}
			\begin{center}
				\begin{tikzpicture}[scale=1.7]
					\node[vertex] (a) at (0,0) {};
					\node[vertex] (b) at (1,0) {};
					\node[vertex] (c) at (0.5,0.866025) {};
					\node[vertex] (d) at (2.5,0) {};
					\node[vertex] (e) at (3.5,0) {};
					\node[vertex] (f) at (3,0.866025) {};
					\node[nvertex] (g) at (3,0.33) {};
					\draw[edgeq] (a)edge(c) (b)edge(c) (d)edge(f) (e)edge(f);
					\draw[oedge] (a)edge(b);
					\draw[nedge]  (d)edge(g) (e)edge(g) (f)edge(g);
					\draw[ultra thick,->] (1.5,0.5) -- (2,0.5);
				\end{tikzpicture}
			\end{center}		
			\caption{Type~II}\label{fig:type2}
		\end{subfigure}
  \end{center}
  \caption{Henneberg moves}
  \label{fig:henneberg}
\end{figure}

In Theorem~\ref{thm:lamanHasNACorProblematic} we state a result that holds for a wide range of Laman graphs.
The following definition classifies those graphs for which it does not apply.
However, for all known examples of such graphs, we can still find a NAC-coloring.
One example of such a graph is given in Figure~\ref{fig:problematic12}.
\begin{defn}\label{def:problematicGraphs}
	A Laman graph $G$ is called \textit{problematic}, if the following hold:
	\begin{enumerate}
		\item  $\deg(v)\geq 3$ for all~$v\in V_G$,
		\item  if $\deg(v) =3$, then exactly two neighbors of~$v$ are connected by an edge and both have degree at least 4,
		\item  all vertices are in some triangle $C_3\subset G$.
	\end{enumerate}
\end{defn}
\begin{figure}[ht]
  \begin{center}
    \begin{tikzpicture}[scale=2]
      \node[vertex] (a1) at (0.5, -0.866025) {};
      \node[vertex] (a2) at (1., 0.) {};
      \node[vertex] (a3) at (0.5, 0.866025) {};
      \node[vertex] (a4) at (-0.5, 0.866025) {};
      \node[vertex] (a5) at (-1.,  0.) {};
      \node[vertex] (a6) at (-0.5, -0.866025) {};
      \node[vertex] (i1) at (0, 0.5) {};
      \node[vertex] (i2) at (-0.433013, 0.25) {};
      \node[vertex] (i3) at (-0.433013, -0.25) {};
      \node[vertex] (i4) at (0, -0.5) {};
      \node[vertex] (i5) at (0.433013, -0.25) {};
      \node[vertex] (i6) at (0.433013, 0.25) {};
      \draw[bedge] (a1)edge(a2) (a2)edge(a3) (a3)edge(a4) (a4)edge(a5) (a1)edge(i5) (i5)edge(a2) (a2)edge(i6) (i6)edge(a3) (a3)edge(i1) (i1)edge(a4) (a4)edge(i2) (i2)edge(a5) (i2)edge(i5);
      \draw[redge] (a5)edge(a6) (a6)edge(a1) (a5)edge(i3) (i3)edge(a6) (a6)edge(i4) (i4)edge(a1) (i3)edge(i6) (i4)edge(i1);
    \end{tikzpicture}
  \end{center}
  \caption{A problematic graph on 12 vertices with a NAC-coloring}
  \label{fig:problematic12}
\end{figure}
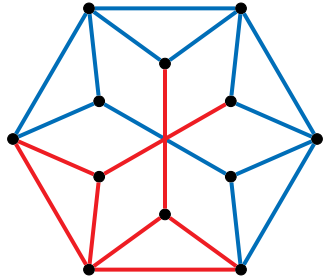

We prove two technical lemmas which are needed in the proof of Theorem~\ref{thm:lamanHasNACorProblematic}.
\begin{lem}\label{lem:undoIIb}
	Let~$G$ be a Laman graph. Let $G$ contain a triangle induced by vertices~$u,v, u_1\in V_G$ such that
	\begin{enumerate}
		\item $\deg_G u =\deg_G v=3$ and $\deg_G u_1>3$,
		\item $wu_2\in E_G$, where~$w\in V_G\setminus \{v, u_1\}$ is the third neighbor of~$u$, and $u_2\in V_G\setminus \{u, u_1\}$ is the third neighbor of~$v$,
		\item $wu_1,u_1u_2\notin E_G$.
	\end{enumerate}
	Then there is a Laman graph $G'=(V',E')$ such that~$G$ is constructed from $G'$ by a Henneberg move IIb by linking the vertex~$v$  to the vertices~$u,u_1,u_2 \in V'$ 
	such that~$uu_1\in E',uu_2\notin E'$ and removing the edge~$u_1u_2\in E'$, or the situation is symmetric by replacing~$v$ by~$u$ and~$w$ by~$u_2$. 
\end{lem}

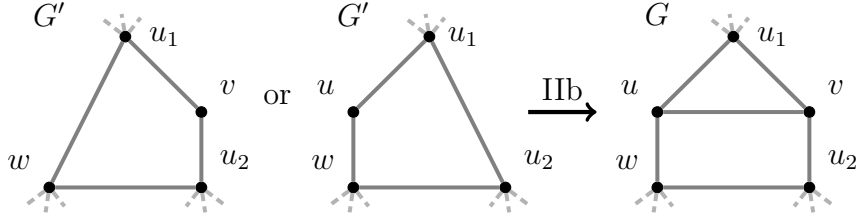
\begin{figure}[ht]
	\centering
	\begin{tikzpicture}[scale=1]
	    \node (G) at (0,2.3) {$G'$};
	    \node[vertex] (w) at (0,0) {};
		\node[vertex] (u2) at (2,0) {};
		\node[vertex] (v) at (2,1) {};
		\node (u) at (0,1) {};
		\node[vertex] (u1) at (1,2) {};			
		\draw[edge] (w) to (u2);	
		\draw[edge] (v) to (u2);
		\draw[edge] (v) to (u1);
		\draw[edge] (w) to (u1);
		\node[above left=0.1cm] at (w) {$w$};
		\node[above right=0.1cm] at (u2) {$u_2$};
		\node[above right=0.1cm] at (v) {$v$};
		\node[right=0.17cm] at (u1) {$u_1$};
		\draw[edgeq] (w) -- +(-0.33,-0.25);
		\draw[edgeq] (w) -- +(-0.05,-0.33);
		\draw[edgeq] (w) -- +(0.2,-0.25);
		\draw[edgeq] (u2) -- +(0.33,-0.25);
		\draw[edgeq] (u2) -- +(0.05,-0.33);
		\draw[edgeq] (u2) -- +(-0.2,-0.25);
		\draw[edgeq] (u1) -- +(-0.32,0.25);
		\draw[edgeq] (u1) -- +(-0.05,0.3);
		\draw[edgeq] (u1) -- +(0.2,0.25);
				
		\node at (3,1.2) {or};		
		
		\node[xshift=4cm] (G) at (G) {$G'$};
		\node[vertex, xshift=4cm] (w) at (w) {};
		\node[vertex, xshift=4cm] (u2) at (u2) {};
		\node[xshift=4cm] (v) at (v) {};
		\node[vertex, xshift=4cm] (u) at (u) {};
		\node[vertex, xshift=4cm] (u1) at (u1) {};			
		\draw[edge] (w) to (u2);	
		\draw[edge] (u) to (w);
		\draw[edge] (u) to (u1);
		\draw[edge] (u2) to (u1);
		\node[above left=0.1cm] at (w) {$w$};
		\node[above right=0.1cm] at (u2) {$u_2$};
		\node[above left=0.1cm] at (u) {$u$};
		\node[right=0.1cm] at (u1) {$u_1$};
		\draw[edgeq] (w) -- +(-0.33,-0.25);
		\draw[edgeq] (w) -- +(-0.05,-0.33);
		\draw[edgeq] (w) -- +(0.2,-0.25);
		\draw[edgeq] (u2) -- +(0.33,-0.25);
		\draw[edgeq] (u2) -- +(0.05,-0.33);
		\draw[edgeq] (u2) -- +(-0.2,-0.25);
		\draw[edgeq] (u1) -- +(-0.32,0.25);
		\draw[edgeq] (u1) -- +(-0.05,0.3);
		\draw[edgeq] (u1) -- +(0.2,0.25);

		\draw [ultra thick,->] (6.3,1) to node[above] {IIb} (7.2,1);
		\node[xshift=4cm] (G) at (G) {$G$};
		\node[vertex, xshift=4cm] (w) at (w) {};
		\node[vertex, xshift=4cm] (u2) at (u2) {};
		\node[vertex, xshift=4cm] (v) at (v) {};
		\node[vertex, xshift=4cm] (u) at (u) {};
		\node[vertex, xshift=4cm] (u1) at (u1) {};			
		\draw[edge] (w) to (u2);	
		\draw[edge] (v) to (u2);
		\draw[edge] (v) to (u1);
		\draw[edge] (w) to (u);
		\draw[edge] (u1) to (u);
		\draw[edge] (v) to (u);
		\node[above left=0.1cm] at (w) {$w$};
		\node[above right=0.1cm] at (u2) {$u_2$};
		\node[above right=0.1cm] at (v) {$v$};
		\node[above left=0.1cm] at (u) {$u$};
		\node[right=0.17cm] at (u1) {$u_1$};
		\draw[edgeq] (w) -- +(-0.33,-0.25);
		\draw[edgeq] (w) -- +(-0.05,-0.33);
		\draw[edgeq] (w) -- +(0.2,-0.25);
		\draw[edgeq] (u2) -- +(0.33,-0.25);
		\draw[edgeq] (u2) -- +(0.05,-0.33);
		\draw[edgeq] (u2) -- +(-0.2,-0.25);
		\draw[edgeq] (u1) -- +(-0.32,0.25);
		\draw[edgeq] (u1) -- +(-0.05,0.3);
		\draw[edgeq] (u1) -- +(0.2,0.25);
	\end{tikzpicture}
	\caption{Undoing Henneberg move IIb}
	\label{fig:undoIIb}
\end{figure}

\begin{proof}
	See Figure~\ref{fig:undoIIb} for the sketch of the situation.
	According to Theorem~6.4 in \cite{Laman1970}, we can revert a Henneberg move II by removing the vertex~$v$ and adding the edge between two of its neighbors, 
	namely~$u_1$ and~$u_2$, if there exists no subgraph $M$ of 
	$G\setminus v=(V_G\setminus \{v\},\{e\in E_G \colon v\notin e\})$ 
	such that~$u_1,u_2\in V_M$ and $|E_M|=2|V_M|-3$. Or the vertex~$u$ can be removed and the edge~$u_1w$ be added
	if there is no subgraph $L$ of $G\setminus u$ such that~$u_1,w\in V_L$ and $|E_L|=2|V_L|-3$.
	
	Assume on the contrary that such subgraphs $M$ and $L$ exist. Obviously, $u\notin V_M$, resp.\ $v\notin V_L$, since otherwise the subgraph $(V_M\cup\{v\})$,
	resp.\ $(V_L\cup\{u\})$, of~$G$ has too many edges. We consider the subgraph $H$ of~$G$ induced by $V_H=V_M\cup V_L \cup \{u,v\}$.
	If $V_M\cap V_L=\{u_1\}$, then
	\begin{align*}
		|E_H|\geq|E_M|+|E_L|+6&=2|V_L|+2|V_M|=2(|V_M \cup V_L|)+2\\
		&=2(|V_M \cup V_L \cup \{u,v\}|)-2=2|V_H|-2\,,
	\end{align*}
	which contradicts Laman condition. If $|V_M\cap V_L|\geq 2$, then the subgraph $M\cup L$ of~$G$ induced by $V_M\cup V_L$ satisfies 
	$|E_{M\cup L}|=2|V_{M\cup L}|-3$ by Lemma~6.2 in  \cite{Laman1970} and we have
	\begin{align*}
		|E_H|&=|E_{M\cup L}|+5=2|V_{M\cup L}|-3+5=2(|V_{M\cup L}|+2)-2\\
		&=2(|V_M \cup V_L \cup \{u,v\}|)-2=2|V_H|-2\,,
	\end{align*}
	which is also a contradiction.
\qed\end{proof}

\begin{lem}\label{lem:TrCompHasAtLeastTwoDegreeTwo}
	If $T$ is a \trcon{} Laman graph, then $T$ has at least two vertices of degree two.
\end{lem}
\begin{proof}
	The claim is trivial for $|V_T|\leq 4$. We proceed by induction on the Henneberg construction.
	The last Henneberg step in the construction of $T$ can only be of type Ia or IIc. If it is Ia,
	then the number of vertices of the degree two obviously remains or increases, because if $|V_T|\geq 4$,
	then there are no adjacent vertices of degree two.

	If the step is of type IIc, let $T$ be constructed from $T'$ so that a new vertex~$v$ is linked by edges to the vertices~$u,u_1,u_2 \in T_{V'}$,
	$uu_1,uu_2\in E_{T'}$ and the edge~$u_1u_2\in E_{T'}$ is removed.
	Only the degree of~$u$ increases, so we assume that $\deg_{T'} u=2$.
	But then $T'$ is just a triangle, otherwise $T$ is not \trcon{}.
	Hence, the number of vertices of $T$ of degree two is still at least two.
\qed\end{proof}

Now we have all necessary tools to show that the conjecture holds for all Laman graphs
which are not constructed by Henneberg steps from problematic ones.
We prove even more, namely, the conjecture is true if we could show that all problematic graphs have a NAC-coloring.
\begin{thm}\label{thm:lamanHasNACorProblematic}
	Let~$G$ be a Laman graph. If~$G$ is not \trcon{}, then it has a NAC-coloring, 
	or there exists a problematic graph $\widetilde{G}$ with no NAC-coloring such that  
	$G$ can be constructed from $\widetilde{G}$ by Henneberg steps (possibly $G=\widetilde{G}$).
\end{thm}
\begin{proof}
	We proceed by induction on the Henneberg construction.
	Let $G'$ be a graph such that~$G$ is obtained from $G'=(V',E')$ by adding a vertex~$v$ by a Henneberg move. 
	We distinguish five cases according to Definition~\ref{def:typesOfHennebergMoves}.

	Ia: Let~$v$ be constructed by a Henneberg move of type Ia from vertices~$u,u'$ and~$uu'\in E_G$.
	The graph $G'$ is \trcon{} if and only if~$G$ is \trcon{}. 
	If $G'$ has a NAC-coloring~$\delta'$, then we define a NAC-coloring~$\delta$ of~$G$ by $\delta|_{E'}=\delta'$ and 
	$\delta(uv)=\delta(u'v):=\delta'(uu')$.
	Otherwise, we have a problematic graph $\widetilde{G}$ from the induction hypothesis.

	Ib and IIa: The vertex~$v$ is not in any triangle, so there is a NAC-coloring by Corollary~\ref{cor:vertexNotInTriangle}.

	IIc: We may assume that there are no vertices of degree two in $G$, otherwise a previous case applies.
	Let~$v$ be constructed so that it is linked by edges to the vertices $u,u_1,u_2 \in V'$ such that~$uu_1,uu_2\in E'$ 
	and the edge~$u_1u_2\in E'$ is removed. If $G'$ is not \trcon{}, 
	then there is a problematic graph $\widetilde{G}$ without any NAC-coloring, or $G'$ has a NAC-coloring~$\delta'$ 
	and we define a NAC-coloring~$\delta$ of~$G$ by $\delta|_{E'}=\delta'$ 
	and $\delta(uv)=\delta(u_1v)=\delta(u_2v):=\delta'(u_1u_2)$. 
	Otherwise, $G'$ has at least two vertices of degree two by Lemma~\ref{lem:TrCompHasAtLeastTwoDegreeTwo}.
	But the degrees of vertices in $G$ and $G'$ are the same except for $u$.
	Hence, there must be a vertex of degree two in $G$, which contradicts our assumption.

	IIb: We may assume that there is no $G'$ from which $G$ can be constructed by Ia,b or IIa,c, i.e.,
	all vertices of~$G$ have degree at least four or can be constructed by IIb.
	Note that if $t$ denotes the number of vertices of degree three, then $t\geq 6$:
	$$
	2|V_G|-3=|E_G|=\frac{1}{2} \sum_{w\in V_G} \deg w\geq \frac{1}{2}\cdot 4(|V_G|-t)+\frac{1}{2}\cdot 3 t=2|V_G|-\frac{t}{2}\,.
	$$
	
	If there are three vertices in a triangle in~$G$ such that all of them have degree three,
	then there is a NAC-coloring by Theorem~\ref{thm:connectingEdgesCut}.

	Let there be a triangle induced by vertices~$u,v, u_1\in V_G$ such that $\deg_G u =\deg_G v=3$ and $\deg_G u_1>3$.
	Let~$w\in V_G\setminus \{v, u_1\}$ be the third neighbor of~$u$ and $u_2\in V_G\setminus \{u, u_1\}$ be the third neighbor of~$v$.
	We remark that the vertices~$w$ and~$u_1$, resp.\ $u_1$ and $u_2$, are nonadjacent, otherwise~$u$, resp. $v$, could be constructed by IIc.
	There are two possibilities: if ~$wu_2\notin E_G$, then $\{u_1,u_2,w\}$ separates~$G$ and hence, there
	is a NAC-coloring by Theorem~\ref{thm:independentVertexCut}.

	If~$wu_2\in E_G$, then Lemma~\ref{lem:undoIIb} justifies that there is a $G'$ with the property that the vertex~$v$ is constructed
	by linking to the vertices~$u,u_1,u_2 \in V'$ such that~$uu_1\in E',uu_2\notin E'$ and the edge~$u_1u_2\in E'$ is removed,
	or the situation is symmetric by replacing~$v$ by~$u$ and~$w$ by~$u_2$. Assume the former one, see Figure~\ref{fig:undoIIbIb}.
	Since the degree of~$u$ in~$G'$ is two, $G'$ can be obtained from $G''$ by adding the vertex~$u$ by Henneberg move Ib. 
	
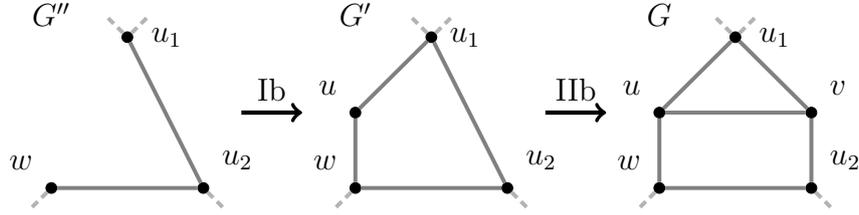
\begin{figure}[ht]
	\centering
	\begin{tikzpicture}[scale=1]
	    \node (G) at (0,2.3) {$G''$};
	    \node[vertex] (w) at (0,0) {};
		\node[vertex] (u2) at (2,0) {};
		\node[] (v) at (2,1) {};
		\node (u) at (0,1) {};
		\node[vertex] (u1) at (1,2) {};	
		\draw[edge] (w) to (u2);
		\draw[edge] (u2) to (u1);
		\node[above left=0.1cm] at (w) {$w$};
		\node[above right=0.1cm] at (u2) {$u_2$};
		\node[right=0.17cm] at (u1) {$u_1$};
		\draw[edgeq] (w) -- +(-0.3,-0.3);
		\draw[edgeq] (u2) -- +(0.3,-0.3);
		\draw[edgeq] (u1) -- +(-0.3,0.3);
		\draw[edgeq] (u1) -- +(0.3,0.3);
				
		\draw [ultra thick,->] (2.5,1) to node[above] {Ib} (3.3,1);		
		
		\node[xshift=4cm] (G) at (G) {$G'$};
		\node[vertex, xshift=4cm] (w) at (w) {};
		\node[vertex, xshift=4cm] (u2) at (u2) {};
		\node[xshift=4cm] (v) at (v) {};
		\node[vertex, xshift=4cm] (u) at (u) {};
		\node[vertex, xshift=4cm] (u1) at (u1) {};			
		\draw[edge] (w) to (u2);	
		\draw[edge] (u) to (w);
		\draw[edge] (u) to (u1);
		\draw[edge] (u2) to (u1);
		\node[above left=0.1cm] at (w) {$w$};
		\node[above right=0.1cm] at (u2) {$u_2$};
		\node[above left=0.1cm] at (u) {$u$};
		\node[right=0.1cm] at (u1) {$u_1$};
		\draw[edgeq] (w) -- +(-0.3,-0.3);
		\draw[edgeq] (u2) -- +(0.3,-0.3);
		\draw[edgeq] (u1) -- +(-0.3,0.3);
		\draw[edgeq] (u1) -- +(0.3,0.3);

		\draw [ultra thick,->] (6.5,1) to node[above] {IIb} (7.3,1);
		\node[xshift=4cm] (G) at (G) {$G$};
		\node[vertex, xshift=4cm] (w) at (w) {};
		\node[vertex, xshift=4cm] (u2) at (u2) {};
		\node[vertex, xshift=4cm] (v) at (v) {};
		\node[vertex, xshift=4cm] (u) at (u) {};
		\node[vertex, xshift=4cm] (u1) at (u1) {};			
		\draw[edge] (w) to (u2);	
		\draw[edge] (v) to (u2);
		\draw[edge] (v) to (u1);
		\draw[edge] (w) to (u);
		\draw[edge] (u1) to (u);
		\draw[edge] (v) to (u);
		\node[above left=0.1cm] at (w) {$w$};
		\node[above right=0.1cm] at (u2) {$u_2$};
		\node[above right=0.1cm] at (v) {$v$};
		\node[above left=0.1cm] at (u) {$u$};
		\node[right=0.17cm] at (u1) {$u_1$};
		\draw[edgeq] (w) -- +(-0.3,-0.3);
		\draw[edgeq] (u2) -- +(0.3,-0.3);
		\draw[edgeq] (u1) -- +(-0.3,0.3);
		\draw[edgeq] (u1) -- +(0.3,0.3);
	\end{tikzpicture}
	\caption{Undoing Henneberg move IIb and Ib (dashed edges illustrate the minimal degrees)}
	\label{fig:undoIIbIb}
\end{figure}

	Since $\deg_G u_1>3$, $\deg_{G} u_2 \geq 3$ and $\deg_G w\geq3$, we have $\deg_{G''} u_1\geq 3$, $\deg_{G''} u_2 \geq 3$ and $\deg_{G''} w\geq2$.
	Assume that $G''$ is \trcon{}. There are at least two vertices of degree two in~$G''$ by Lemma~\ref{lem:TrCompHasAtLeastTwoDegreeTwo}. 
	Hence, there is a vertex~$v'\in V_{G''}\setminus\{u_1,u_2,w\}$ of degree two.
	But $\deg_{G''}v'=\deg_G v'$ which contradicts that all vertices of~$G$ have degree at least three.
	Thus, $G''$ is not \trcon{}.

	By the induction hypothesis, there is a problematic graph $\widetilde{G}$ without any NAC-coloring
	or $G''$ has a NAC-coloring~$\delta''$. 
	If the latter holds, we obtain a NAC-coloring~$\delta$ of~$G$ from~$\delta''$ in the following way: 
	$\delta|_{E''}=\delta''$, $\delta(uw):=\delta''(wu_2)$ 
	and $\delta(u_1v)=\delta(u_2v)=\delta(uu_1)=\delta(uv):=\delta''(u_1u_2)$.
	There is no almost \blue{} or almost \red{} cycle in~$G$, since its existence would contradict that~$\delta''$ is a NAC-coloring of $G''$.

	The remaining graphs are such that all vertices have degree at least three and if~$v$ has degree three,
	then precisely two neighbors~$u,u_1$ of~$v$ are adjacent and have degree at least four.
	If~$G$ has a vertex which is not in any triangle, then there is a NAC-coloring by Corollary~\ref{cor:vertexNotInTriangle}. 
	Hence, $\widetilde{G}=G$ is problematic for the remaining cases.
\qed\end{proof}

\section{Flexible Labelings in Dimension Three}	\label{sec:threeDim}
It turns out that the question of existence of a flexible labeling for a given graph is much easier
when we consider straightforward modification of the definitions replacing $\RR^2$ by $\RR^3$.
We get the following result.
\begin{prop}
	If $G$ is a graph that is not complete, then there exists a labeling with infinitely many non-equivalent compatible realizations in $\RR^3$.
\end{prop}
\begin{proof}
	Let $u$ and $v$ be two nonadjacent vertices of $G$.
	Let $\alpha \in [0,2\pi]$.
	We define a realization $\rho$ by $\rho(u)=(1,0,0), \rho(v)=(\cos \alpha,1,\sin\alpha)$ and $\rho(w)=(0,y_w,0)$ for $w\in V_G\setminus\{u,v\}$,
	where $y_w$ are pairwise distinct real numbers.
	The induced labeling is independent of $\alpha$ and therefore flexible.
\qed\end{proof}
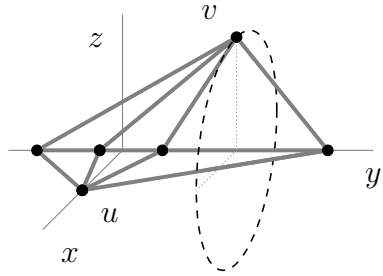
\begin{figure}[ht]
	\centering
	\begin{tikzpicture}[scale=1.5]
	\node[below right=0.1cm] at (-0.7,-0.7) {$x$};
	\node[below=0.1cm] at (2.2,0) {$y$};
	\node[below left=0.1cm] at (0,1.2) {$z$};
	\draw[gray] (0,0) to (0,1.2) {};
	\draw[gray] (-1,0) to (2.2,0) {};
	\draw[gray] (0,0) to (-0.7,-0.7) {};
	\node[vertex] (p) at (-0.2,0) {};
	\node[vertex] (w) at (1.8,0) {};
	\node[vertex] (w2) at (-0.75,0) {};
	\node[vertex] (u) at (-0.353,-0.353) {};
	\node[vertex] (v) at (1,1) {};
	\node[vertex] (ut) at (0.35,0) {};
	\coordinate (s) at (1,0);
	\draw[gray, densely dotted] (s) to (v) {};	
	\draw[gray, densely dotted] (s) to (1-0.353,-0.353) {};	
	
	\draw[edge] (p) to (v);
	\draw[edge] (ut) to (u);
	\draw[edge]  (v) to (w2);		
	\draw[edge] (u) to (w) (u) to (w2);
	\draw[edge] (p) to (u) (p) to (w2);	
	\draw[edge] (ut) to (s);
	\draw[edge] (p) to (ut);
	\draw[edge] (v) to (ut);		
	
	\node[vertex] at (u) {};
	\node[vertex] at (p) {};		

	\draw[dashed, line width=0.6pt, yslant=1] (s) ellipse (0.353cm and 1cm);
	\node[below right=0.1cm] at (u) {$u$};
	\node[above left=0.1cm] at (v) {$v$};

	\draw[edge] (v) to (w)  (s) to (w);
	\draw[edge] ($(u)!0.5!(w)$) to (w);

	\end{tikzpicture}
	\caption{Flexible labeling in $\RR^3$ of a graph with nonadjacent vertices $u$ and $v$}
\end{figure}

\section{Conclusion}
The main result of our work is the combinatorial characterization of existence of a flexible labeling using the concept of NAC-coloring.
Based on this characterization, we provide some necessary and sufficient conditions on a graph to have a NAC-coloring.
We also conjectured a characterization of Laman graphs with a flexible labeling and proved the conjecture for a wide class of graphs.
Laman graphs are interesting class to study since every graph which is not flexible for a generic choice of lengths must have a Laman graph as a subgraph.
We think that NAC-colorings are a good tool for determining non-generic flexible labelings.

Additionally to the conjecture, we are interested in further research questions: can we characterize graphs which have a flexible labeling with injective realizations?
Can we say more about a flexible labeling from a NAC-coloring besides existence? 
For instance, we observed that for a given NAC-coloring, some algebraic equations for lengths of edges can be derived.
We mentioned that an irreducible curve of realizations might give more than one NAC-coloring.
Having the set of all NAC-colorings of a graph, can we determine subsets corresponding to an irreducible curve of realizations?
Is there any procedure which yields all possible flexible modes?

\end{document}